\documentclass[11pt, a4paper, reqno]{amsart}

\usepackage{amsmath, amssymb, amsthm}
\usepackage[T1]{fontenc}
\usepackage[utf8]{inputenc}
\usepackage[all]{xy}
\usepackage[
left=2.5cm,right=2.5cm,top=2.5cm,bottom=2.5cm,a4paper
  ]{geometry}
\usepackage{hyperref}
\usepackage{indentfirst}
\usepackage{mathtools}
\usepackage{enumitem}
\usepackage{float}
\usepackage{multirow}
\usepackage{makecell}
\usepackage{array,longtable}

\usepackage[last]{changes}
\definechangesauthor[color=blue]{KY}

\DeclareMathOperator{\Gal}{Gal}

\def\C{{\mathbb C}}
\def\D{{\mathbb D}}

\def\F{{\mathbb F}}

\def\Q{{\mathbb Q}}
\def\R{{\mathbb R}}

\def\Z{{\mathbb Z}}

\newcommand{\wbar}[1]{\overline{#1}}

\newcommand{\bs}[1]{\left({#1}\right)}
\newcommand{\bm}[1]{\left\{{#1}\right\}}
\newcommand{\bl}[1]{\left[{#1}\right]}

\newcommand{\abs}[1]{\left|#1\right|}
\newcommand{\subgp}[1]{\left\langle#1\right\rangle}
\newcommand{\psmatrix}[1]{\bs{\begin{smallmatrix} #1 \end{smallmatrix}}}

\newcommand{\Union}[1]{{\bigsqcup\limits_{\substack{#1}}}}

\newcommand{\diag}[2][ ]{{\Delta^{#1}_{\bs{#2}}}}

\newcommand{\conj}[2]{\prescript{#1}{}{#2}}

\newcommand{\ep}{\epsilon}

\newcommand{\z}[1]{\zeta_{#1}}

\newcommand{\Tr}[1]{\mathrm{Tr}\bs{#1}}

\newcommand{\wtilde}[1]{\widetilde{#1}}

\newcommand{\fka}{\mathfrak{a}}
\newcommand{\fkb}{\mathfrak{b}}

\newcommand{\calB}{\mathcal{B}}

\theoremstyle{plain}
\newtheorem{theorem}{Theorem}[section]
\newtheorem{lemma}[theorem]{Lemma}
\newtheorem{proposition}[theorem]{Proposition}
\newtheorem{corollary}[theorem]{Corollary}
\newtheorem{remark}[theorem]{Remark}
\newtheorem{definition and lemma}[theorem]{Definition and Lemma}
\newtheorem{problem}[theorem]{Problem}

\theoremstyle{definition} 
\newtheorem{definition}[theorem]{Definition}

\newtheorem{example}[theorem]{Example}

\theoremstyle{remark} 

\DeclareFontFamily{U}{wncy}{}
\DeclareFontShape{U}{wncy}{m}{n}{<->wncyr10}{}
\DeclareSymbolFont{mcy}{U}{wncy}{m}{n}
\DeclareMathSymbol{\Sha}{\mathord}{mcy}{"58}

\title[Automorphism groups of polarized abelian threefolds]{A classification of the automorphism groups of polarized abelian threefolds over finite fields}
\author{WonTae Hwang, Bo-Hae Im and Hansol Kim}
\address{School of Mathematics, Korea Institute for Advanced Study, 85 Hoegiro, Dongdaemun-gu, Seoul 02455, South Korea}
\email{hwangwon@kias.re.kr}
\address{Dept. of Mathematical Sciences, KAIST, 291 Daehak-ro, Yuseong-gu, Daejeon, South Korea, 34141}
\email{bhim@kaist.ac.kr}
\address{Dept. of Mathematical Sciences, KAIST, 291 Daehak-ro, Yuseong-gu, Daejeon, South Korea, 34141}
\email{jawlang@kaist.ac.kr}

\begin{document}



\maketitle

\begin{abstract}
We give a classification of maximal elements of the set of finite groups that can be realized as the automorphism groups of polarized abelian threefolds over finite fields.
\end{abstract}

\section{Introduction}
Let $k$ be a field, and let $X$ be an abelian threefold over $k.$ We denote the endomorphism ring of $X$ over $k$ by $\textrm{End}_k(X)$. It is a free $\Z$-module of rank $\leq 36.$ We also let $\textrm{End}^0_k(X)=\textrm{End}_k(X) \otimes_{\Z} \Q.$ This $\Q$-algebra $\textrm{End}_k^0(X)$ is called the endomorphism algebra of $X$ over $k.$ Then $\textrm{End}_k^0(X)$ is a finite dimensional semisimple algebra over $\Q$ with $6 \leq \textrm{dim}_{\Q} \textrm{End}_k^0(X) \leq 36.$ Moreover, if $X$ is $k$-simple, then $\textrm{End}_k^0(X)$ is a division algebra over $\Q$. Now, it is a well-known fact that $\textrm{End}_k (X)$ is a $\Z$-order in $\textrm{End}_k^0(X),$ and the group $\textrm{Aut}_k(X)$ of the automorphisms of $X$ over $k$ is not finite, in general. But if we fix a polarization $\mathcal{L}$ on $X$, then the group $\textrm{Aut}_k(X,\mathcal{L})$ of the automorphisms of the polarized abelian threefold $(X,\mathcal{L})$ over $k$ is always finite. In this regard, it might be interesting to consider the following problems.
\begin{problem}\label{main prob}
(1) Classify all (finite) groups $G$ (up to isomorphism) such that there exist a field $k$ and an abelian variety $X$ of dimension $3$ over $k$ with $G \leq  \textrm{Aut}_k(X)$.\\
(2) Classify all finite groups $G$ (up to isomorphism) such that there exist a field $k$ and an abelian variety $X$ of dimension $3$ over $k$ with $G = \textrm{Aut}_k(X,\mathcal{L})$ for some polarization $\mathcal{L}$ on $X.$
\end{problem}
We note that the answer for Problem~\ref{main prob}~(2) gives a (partial) answer for Problem~\ref{main prob}~(1). For the case when $k=\C,$ Birkenhake, Gonz$\acute{\textrm{a}}$lez and Lange \cite{BGL(1999)} computed all finite automorphism groups of complex tori of dimension 3, which are maximal in the isogeny class. The goal of this paper is to give an almost complete answer for Problem~\ref{main prob}~(2) for the case when $k$ is a finite field, by classifying all finite groups $G$ that can be realized as the automorphism group $\textrm{Aut}_k(X,\mathcal{L})$ of a polarized abelian threefold $(X,\mathcal{L})$ over a finite field $k,$ which are {\it{maximal}} in the following sense: there is no finite group $H$ such that $G$ is isomorphic to a proper subgroup of $H$ and $H=\textrm{Aut}_k (Y,\mathcal{M})$ for some abelian threefold $Y$ over $k$ that is $k$-isogenous to $X$ with a polarization $\mathcal{M}.$\\

Along this line, the first author gave a classification of maximal automorphism groups for arbitrary polarized abelian surfaces over finite fields in \cite{4}, and provided a complete list of finite groups that can be realized as the automorphism groups of simple polarized abelian varieties of odd prime dimension over finite fields in \cite{5}. In this paper, as the next step toward the goal of completing the three dimensional case, we will give such a classification for arbitrary polarized abelian threefolds over finite fields. Two main difficulties in achieving the goal are as follows: 1) there are way more cases to be considered according to the possible decomposition of abelian threefolds (compared to those of \cite{4}), and 2) we need to deal with finite subgroups of rings of matrices of larger dimension over real quadratic fields and over totally definite quaternion algebras over $\mathbb{Q}$ (compared to those of \cite{4,5}).       \\

Now, our main result is summarized in the following theorem.
\begin{theorem}\label{main theorem}
  The possibilities for maximal automorphism groups of a polarized abelian threefold $(X, \mathcal{L})$ over a finite field are given by the lists in Theorems~\ref{thm old 24},~\ref{prodnonisoellip},~\ref{prodnonisoellip2},~\ref{prodnonisoellip3},~\ref{powordelli}, and~\ref{pow of supell} in the case of $X$ being simple, a product of a simple abelian surface and an elliptic curve, a product of three pairwise non-isogenous elliptic curves, a product of a power of an elliptic curve and an elliptic curve that are non-isogenous, a power of an ordinary elliptic curve, and a power of a supersingular elliptic curve, respectively.
\end{theorem}
To obtain such a classification, we need to combine various mathematical areas, including the theory of abelian varieties over finite fields, finite subgroups of division rings, and quaternionic representation theory. For more detailed statements and proofs, see Theorems~\ref{thm old 24},~\ref{prodnonisoellip},~\ref{prodnonisoellip2},~\ref{prodnonisoellip3},~\ref{powordelli}, and~\ref{pow of supell} below. This kind of classification result might not only be interesting in its own sake, but also, have some applications to other areas of mathematics, such as group theory computing the exact values of Jordan constants of the automorphism groups of certain abelian varieties over fields of positive characteristic (see \cite{Hwa(2020)}), and (potentially) cryptography applying the classification to discuss the rationality of geometrically rational generalized Kummer surfaces over a finite field (see \cite{Kos(2021)}).\\

This paper is organized as follows: In Section~\ref{prelim}, we introduce several facts which are related to our desired classification. Explicitly, we will recall some facts about endomorphism algebras of abelian varieties ($\S$\ref{end alg av}), the theorem of Tate ($\S$\ref{thm Tate sec}), Honda-Tate theory ($\S$\ref{thm Honda}), and maximal orders over a Dedekind domain ($\S$\ref{max ord dede}). In Section~\ref{findiv}, we give a classification of all finite subgroups of the multiplicative subgroup of certain division algebras over $\Q$. In Section~\ref{gl3 sec}, we describe all the maximal finite subgroups of $GL_3(K)$ for the case when either $K=\Q$ or $K$ is a real quadratic field. In Section~\ref{quat mat rep}, we record some useful results on quaternionic matrix groups that are related to our classification, following a paper of Nebe \cite{8}. In Section~\ref{main}, we finally obtain the desired classification using the facts that were introduced in the previous sections. \\

In the sequel, let $q=p^a$ for some prime number $p$ and an integer $a \geq 1,$ unless otherwise stated. Also, for an integer $n \geq 1,$ let $\varphi(n)$ ($C_n$, $D_n$, respectively) denote the number of integers that are smaller than or equal to $n$ and relatively prime to $n$ (a cyclic group of order $n$, a dihedral group of order $2n$, respectively). Finally, for a matrix $A \in GL_2 (\C),$ $\widetilde{A}$ denotes the block matrix $\psmatrix{A & \bf{0} \\ \bf{0} & \det(A)^{-1}} \in GL_{3}(\C),$ and for an $\alpha \in \textrm{Sym}_3,$ $I_{3,\alpha} \in GL_3(\C)$ denotes the permutation matrix corresponding to $\alpha.$

\section{Preliminaries}\label{prelim}

In this section, we briefly recall some of the facts in the general theory of abelian varieties over a field and maximal orders over a Dedekind domain. Our main references are \cite{2}, \cite{7}, and \cite{9}.

\subsection{Endomorphism algebras of abelian varieties}\label{end alg av}
In this section, we give some basic facts about the endomorphism algebra of an abelian variety over a field. To this aim, let $X$ be an abelian variety over a field $k.$ Then the set $\textrm{End}_k (X)$ of endomorphisms of $X$ over $k$ is actually a ring. Because it is more difficult to deal with $\textrm{End}_k(X)$ itself, in general, if we want to work with $\Q$-algebras instead, then we define $\textrm{End}_k^0(X):=\textrm{End}_k (X) \otimes_{\Z} \Q.$ The $\Q$-algebra $\textrm{End}_k^0 (X)$ is called the \emph{endomorphism algebra} of $X$ over $k.$ \\
\indent If $X$ is a simple abelian variety over $k,$ then $\textrm{End}_k^0(X)$ is a division algebra over $\Q.$ If $X$ is an arbitrary abelian variety over $k,$ then it is well known that there exist simple abelian varieties $Y_1,\cdots,Y_n$ over $k,$ no two of which are $k$-isogenous, and positive integers $m_1,\cdots,m_n$ such that $X$ is $k$-isogenous to $Y_1^{m_1} \times \cdots \times Y_n^{m_n}.$ In this situation, we have
\begin{equation*}
\textrm{End}_k^0 (X) \cong M_{m_1}(D_1) \times \cdots \times M_{m_n}(D_n)
\end{equation*}
where $D_i :=\textrm{End}_k^0 (Y_i)$ for each $i.$ Note that each $D_i$ is a division algebra over $\Q.$ Moreover, $\textrm{End}_k^0 (X)$ is a finite dimensional semisimple $\Q$-algebra of dimension at most $4 \cdot (\textrm{dim}~X)^2$ (see \cite[Corollary 12.11]{2}).

\subsection{The theorem of Tate}\label{thm Tate sec}
In this section, we recall a theorem of Tate and a description on the structure of the endomorphism algebra of simple abelian varieties over finite fields, which play an important role throughout the paper. \\



We first recall that an abelian variety $X$ over a field $k$ is called \emph{elementary} if $X$ is $k$-isogenous to a power of a simple abelian variety over $k.$ Then Tate \cite{10} obtained the following result for the case when the base field $k$ is finite.
\begin{theorem}[{\cite[Theorem 2]{10}}]\label{cor TateEnd0}
  Let $X$ be an abelian variety of dimension $g$ over a finite field $k=\F_q$, and let $\pi_X$ be the Frobenius endomorphism of $X.$ Then we have:\\
  (a) The center of $\textrm{End}_k^0(X)$ is the subalgebra $\Q[\pi_X].$ In particular, $X$ is elementary if and only if $\Q[\pi_X]=\Q(\pi_X)$ is a field, and this occurs if and only if $f_X$ is a power of an irreducible polynomial in $\Q[t]$ where $f_X$ denotes the characteristic polynomial of $\pi_X.$ \\
 (b) Suppose that $X$ is elementary. Let $h=f_{\Q}^{\pi_X}$ be the minimal polynomial of $\pi_X$ over $\Q$. Further, let $d=[\textrm{End}_k^0(X):\Q(\pi_X)]^{\frac{1}{2}}$ and $e=[\Q(\pi_X):\Q].$ Then $de =2g$ and $f_X = h^d.$ \\
(c) We have $2g \leq \textrm{dim}_{\Q} \textrm{End}^0_k (X) \leq (2g)^2$ and $X$ is of CM-type. \\
(d) The following conditions are equivalent: \\
  \indent (d-1) $\textrm{dim}_{\Q} \textrm{End}_k^0(X)=2g$; \\
  \indent (d-2) $\textrm{End}_k^0(X)=\Q[\pi_X]$; \\
  \indent (d-3) $\textrm{End}_k^0(X)$ is commutative; \\
  \indent (d-4) $f_X$ has no multiple root. \\
(e) The following conditions are equivalent: \\
  \indent (e-1) $\textrm{dim}_{\Q} \textrm{End}_k^0(X)=(2g)^2$; \\
  \indent (e-2) $\Q[\pi_X]=\Q$; \\
  \indent (e-3) $f_X$ is a power of a linear polynomial; \\
  \indent (e-4) $\textrm{End}^0_k(X) \cong M_g(D_{p,\infty})$ where $D_{p,\infty}$ is the unique quaternion algebra over $\Q$ that is ramified at $p$ and $\infty$, and split at all other primes; \\
  \indent (e-5) $X$ is isogenous to $E^g$ for a supersingular elliptic curve $E$ over $k$ all of whose endomorphisms are defined over $k.$
\end{theorem}

Now, in view of Theorem~\ref{cor TateEnd0}~(a), if $X$ is an elementary abelian variety over a finite field $k$, then $\textrm{End}_k^0(X)$ is a simple algebra over its center $\Q[\pi_X].$ For a more precise description on $\textrm{End}_k^0(X),$ we record the following two results.
\begin{proposition}[{\cite[Corollary 16.30]{2}}]\label{local inv}
  Let $X$ be an elementary abelian variety over a finite field $k=\F_q.$ Let $K=\Q[\pi_X].$ If $\nu$ is a place of $K$, then the local invariant of $\textrm{End}_k^0(X)$ in the Brauer group $\textrm{Br}(K_{\nu})$ is given by
  \begin{equation*}
    \textrm{inv}_{\nu}(\textrm{End}_k^0(X))=\begin{cases} 0 & \mbox{if $\nu$ is a finite place not above $p$}; \\ \frac{\textrm{ord}_{\nu}(\pi_X)}{\textrm{ord}_{\nu}(q)} \cdot [K_{\nu}:\Q_p] & \mbox{if $\nu$ is a place above $p$}; \\ \frac{1}{2} & \mbox{if $\nu$ is a real place of $K$}; \\ 0 & \mbox{if $\nu$ is a complex place of $K$}. \end{cases}
  \end{equation*}
\end{proposition}

\begin{proposition}[{\cite[Corollary 16.32]{2}}]\label{index end alg}
  Let $X$ be a simple abelian variety over a finite field $k.$ Let $d$ be the index of the division algebra $D:=\textrm{End}_k^0(X)$ over its center $\Q(\pi_X)$ (so that $d=[D:\Q(\pi_X)]^{\frac{1}{2}}$ and $f_X = (f_{\Q}^{\pi_X})^d$). Then $d$ is the least common denominator of the local invariants $i_{\nu} = \textrm{inv}_{\nu}(D).$
\end{proposition}

\subsection{Abelian varieties up to isogeny and Weil numbers: Honda-Tate theory}\label{thm Honda}
In this section, we recall an important theorem of Honda and Tate with the following definition.
\begin{definition}\label{qWeil Def}
  (a) A \emph{$q$-Weil number} is an algebraic integer $\pi$ such that $| \iota(\pi) | = \sqrt{q}$ for all embeddings $\iota \colon \Q[\pi] \hookrightarrow \C.$ \\
  (b) Two $q$-Weil numbers $\pi$ and $\pi^{\prime}$ are said to be \emph{conjugate} if they have the same minimal polynomial over $\Q,$ or equivalently, there is an isomorphism $\Q[\pi] \rightarrow \Q[\pi^{\prime}]$ sending $\pi$ to $\pi^{\prime}.$
\end{definition}



\begin{theorem}[{\cite[Main Theorem]{3} or \cite[$\S16.5$]{2}}]\label{thm HondaTata}
 For every $q$-Weil number $\pi$, there exists a simple abelian variety $X$ over $\F_q$ such that $\pi_X$ is conjugate to $\pi$, where $\pi_X$ denotes the Frobenius endomorphism of $X.$ Moreover, we have a bijection between the set of isogeny classes of simple abelian varieties over $\F_q$ and the set of conjugacy classes of $q$-Weil numbers given by $X \mapsto \pi_X$.
\end{theorem}
The inverse of the map $X \mapsto \pi_X$ associates to a $q$-Weil number $\pi$ a simple abelian variety $X$ such that $f_X$ is a power of the minimal polynomial $f_{\Q}^{\pi}$ of $\pi$ over $\Q.$

\subsection{Maximal orders over a Dedekind domain}\label{max ord dede}
In this section, we review the general theory of maximal orders over a Dedekind domain that will be used later in this paper. \\
\indent Throughout this section, let $R$ be a noetherian integral domain with the quotient field $K,$ and let $A$ be a finite dimensional $K$-algebra. Recall that a maximal $R$-order in $A$ is an $R$-order which is not properly contained in any other $R$-order in $A.$ For our later use, we introduce several results about maximal orders.
\begin{theorem}[{\cite[Theorem 8.7]{9}}]\label{mat max}
  Let $A$ be a finite dimensional $K$-algebra. If $\Lambda$ is a maximal $R$-order in $A,$ then for each $n \geq 1,$ $M_n(\Lambda)$ is a maximal $R$-order in $M_n(A).$ If $R$ is integrally closed, then $M_n(R)$ is a maximal $R$-order in $M_n(K).$
\end{theorem}

If we impose additional conditions that $R$ is integrally closed and $A$ is a separable $K$-algebra, then we have the following result saying that a decomposition of $A$ into simple components yields a decomposition of maximal orders in $A$.
\begin{theorem}[{\cite[Theorem 10.5]{9}}]\label{max gen}
  Let $A$ be a separable $K$-algebra with simple components $\{A_i\}_{1 \leq i \leq t}$ and let $R_i$ be the integral closure of $R$ in the center $K_i$ of $A_i$ for each $i.$ Then we have: \\
  (a) For each maximal $R$-order $\Lambda$ in $A,$ we have $\displaystyle \Lambda = \bigoplus_{i=1}^{t} \Lambda e_i$ where $\{e_i\}_{1 \leq i \leq t}$ are the central idempotents of $A$ such that $A_i = A e_i$ for each $i.$ Moreover, each $\Lambda e_i$ is a maximal $R$-order in $A_i=A e_i.$  \\
  (b) If $\Lambda_i $ is a maximal $R$-order in $A_i$ for each $i,$ then $\displaystyle \bigoplus_{i=1}^t \Lambda_i$ is a maximal $R$-order in $A.$ \\
  (c) An $R$-order $\Lambda_i $ in $A_i$ is a maximal $R$-order if and only if $\Lambda_i$ is a maximal $R_i$-order in $A_i.$ \\
\indent (Here, the symbol $\bigoplus$ denotes the (external) direct sum.)
\end{theorem}

Finally, we further assume that $R$ is a Dedekind domain with its quotient field $K \ne R$. Let $A$ be a separable $K$-algebra (which is simple). In the following last theorem, we determine all maximal $R$-orders in $A.$
\begin{theorem}[{\cite[Theorem 21.6]{9}}]\label{mat max 2}
  Let $A=\textrm{Hom}_{D}(V,V) \cong M_r(D)$ be a simple algebra, where $V$ is a right vector space of dimension $r$ over a division algebra $D$ with center $K$. Let $\Delta$ be a fixed maximal $R$-order in $D$, and let $M$ be a full right $\Delta$-lattice in $V.$ Then $\textrm{Hom}_{\Delta}(M,M)$ is a maximal $R$-order in $A.$ If $\Lambda^{\prime}$ is a maximal $R$-order in $A,$ then there is a full right $\Delta$-lattice $N$ in $V$ such that $\Lambda^{\prime}=\textrm{Hom}_{\Delta}(N,N).$
\end{theorem}


\section{Finite Subgroups of Division Algebras}\label{findiv}
In this section, we give a classification of all possible finite groups that can be embedded in the multiplicative subgroup of a division algebra over $\Q$ with certain properties that are related to our situation later in Section~\ref{quat mat rep}. Our main reference is a paper of Amitsur~\cite{1}. We start with the following notion.

\begin{definition}\label{def 9}
  Let $m, r$ be two relatively prime positive integers, and we put $s:=\gcd(r-1, m)$ and $t:=\frac{m}{s}.$ Also, let $n$ be the smallest integer such that $r^n \equiv 1~(\textrm{mod}~m).$ We denote by $G_{m,r}$ the group generated by two elements $a,b$ satisfying the relations
  \begin{equation*}
    a^m=1,~b^n=a^t,~ bab^{-1}=a^r.
  \end{equation*}
  This type of groups includes the \emph{dicyclic group} of order $mn,$ in which case, we often write $\textrm{Dic}_{mn}$ for $G_{m,r}$. As a convention, if $r=1,$ then we put $n=s=1,$ and hence, $G_{m,1}$ is a cyclic group of order $m.$
\end{definition}

Given $m,r,s,t,n,$ as above, we will consider the following two conditions in the sequel: \\
  (C1) $\gcd(n,t)=\gcd(s,t)=1.$ \\
  (C2) $n=2n^{\prime}, m=2^{\alpha} m^{\prime}, s=2 s^{\prime}$ where $\alpha \geq 2,$ and $n^{\prime}, m^{\prime}, s^{\prime}$ are all odd integers. Moreover, $\gcd(n,t)=\gcd(s,t)=2$ and $r \equiv -1~(\textrm{mod}~2^{\alpha}).$ \\
\indent Now, let $p$ be a prime number that divides $m.$ We define: \\
(i) $\alpha_p$ is the largest integer such that $p^{\alpha_p}~|~m.$ \\
(ii) $n_p$ is the smallest integer satisfying $r^{n_p} \equiv 1~(\textrm{mod}~mp^{-\alpha_p}).$ \\
(iii) $\delta_p$ is the smallest integer satisfying $p^{\delta_p} \equiv 1~(\textrm{mod}~mp^{-\alpha_p}).$ \\

Then we have the following theorem that provides us with a useful criterion for a group $G_{m,r}$ to be embedded in a division ring.
\begin{theorem}[{\cite[Theorem 3, Theorem 4, and Lemma 10]{1}}]\label{thm 11}
  A group $G_{m,r}$ can be embedded in a division ring if and only if either (C1) or (C2) holds, and one of the following conditions holds: \\
  (1) $n=s=2$ and $r \equiv -1~(\textrm{mod}~m)$. \\
  (2) For every prime number $q~|~n,$ there exists a prime number $p~|~m$ such that $q \nmid n_p$ and that either \\
  (a) $p \ne 2$, and $\gcd(q,(p^{\delta_p}-1)/s)=1$, or \\
  (b) $p=q=2,$ (C2) holds, and $m/4 \equiv \delta_p \equiv 1~(\textrm{mod}~2).$
\end{theorem}

Now, let $G$ be a finite group. One of our main tools in this section is the following result.
\begin{theorem}[{\cite[Theorem 7]{1}}]\label{thm 12}
  $G$ can be embedded in a division ring if and only if $G$ is of one of the following types: \\
  (1) Cyclic groups. \\
  (2) $G_{m,r}$ where the integers $m,r,$ etc, satisfy Theorem~\ref{thm 11} (which is not cyclic). \\
  (3) $\mathfrak{T}^* \times G_{m,r}$ where $\mathfrak{T}^*$ is the binary tetrahedral group of order $24$, and $G_{m,r}$ is either cyclic of order $m$ with $\gcd(m,6)=1$, or of type (2) with $\gcd(|G_{m,r}|, 6)=1$ (where $|G_{m,r}|$ denotes the order of the group $G_{m,r}).$ In both cases, for all primes $p~|~m,$ the smallest integer $\gamma_p$ satisfying $2^{\gamma_p} \equiv 1~(\textrm{mod}~p)$ is odd. \\
  (4) $\mathfrak{O}^*$, the binary octahedral group of order $48$. \\
  (5) $\mathfrak{I}^*,$ the binary icosahedral group of order $120.$
\end{theorem}

Having stated most of the necessary results, we can proceed to achieve the goal of this section. First, we give two important lemmas, both of whose proofs follow from Theorem~\ref{thm 11}, unless otherwise stated.

\begin{lemma}\label{lem 15}
  Let $m,r,s,t,n$ be as in Definition~\ref{def 9} satisfying (C1) with $n=2$ and $m \in \{2,3,4,6,7,9,14,18\}.$ Then the group $G:=G_{m,r}$ can be embedded in a division ring if and only if $G$ is one of the following groups: \\
  (1) A cyclic group $C_4$\footnote{If we want $G_{m,r}$ to be non-cyclic as in Theorem~\ref{thm 12}, then we may exclude this case.}; \\
  (2) $\textrm{Dic}_{12}$, a dicyclic group of order $12$; \\
  (3) $\textrm{Dic}_{28}$, a dicyclic group of order $28$; \\
  (4) $\textrm{Dic}_{36}$, a dicyclic group of order $36.$
\end{lemma}
\begin{proof}
  Note that $t$ is an odd integer since $n=2$ and $\gcd(n,t)=1.$ Hence, we have the following three cases to consider: \\
\indent (I) If $m \in \{3,7\},$ then $|G|=2m$ is square-free. Now, since $n=2$ and $m$ is an odd prime number, it is easy to see that $r \equiv -1~(\textrm{mod}~m)$ (by the definition of $n$), and then by looking at the presentation of $G,$ we can see that $G$ is not cyclic. By \cite[Corollary 5]{1}, $G$ cannot be embedded in a division ring. \\
\indent (II) Suppose that $m=9.$ Then since $\gcd(s,t)=1,$ we cannot have $s=t=3.$ Also, since $n=2$, we cannot have $s=9, t=1$ by the definition of $n.$ Hence, the only possible case is that $s=1, t=9.$ Since $n=2$ and $\gcd(r+1,r-1) \leq 2,$ we have $r \equiv -1~(\textrm{mod}~9)$ so that the order of $\overline{r}$ in the unit group $U(\Z/9\Z)$ is exactly $2.$ By \cite[Theorem 3.1]{6}, $G$ cannot be embedded in a division ring. \\
\indent (III) If $m \in \{2,4,6,14,18\},$ then we have the following five subcases to consider: \\
  (i) If $m=2,$ then since $\gcd(n,t)=1,$ we get $s=2, t=1.$ In particular, we have $n=s=2$ and $r \equiv -1~(\textrm{mod}~2).$ Hence, $G$ can be $G_{2,r}=C_4.$ \\
  (ii) If $m=4,$ then since $\gcd(n,t)=1,$ we get $s=4, t=1.$ By the definition of $s$, we have $m=4~|~r-1$, and hence, this case cannot occur because of our assumption that $n=2.$ \\
  (iii) If $m=6,$ then since $\gcd(n,t)=1,$ either $s=6, t=1$ or $s=2, t=3.$ If $s=6, t=1,$ then by a similar argument as in (ii), this case cannot occur. Now, if $s=2, t=3,$ then since $n=2,$ we have $r \equiv -1~(\textrm{mod}~6).$ Hence, $G$ can be $G_{6,r}=\textrm{Dic}_{12}.$ \\
  (iv) Similarly, if $m=14,$ then we only need to consider the case when $s=2, t=7.$ Since $n=2,$ we have $r \equiv -1~(\textrm{mod}~14).$ Hence, $G$ can be $G_{14,r}=\textrm{Dic}_{28}.$ \\
  (v) If $m=18,$ then, since $n=2,$ we can see that $r \equiv -1 \pmod{18}$ by a direct computation. Then it follows that $s=\gcd(r-1, 18)=2$ and $t=9.$ As before, $G$ can be $G_{18,r}=\textrm{Dic}_{36}.$ \\
\indent This completes the proof.
\end{proof}

\begin{lemma}\label{lem 16}
  Let $m,r,s,t,n$ be as in Definition~\ref{def 9} satisfying (C2) with $n=2$ and $m \in \{2,3,4,6,7,9,14,18\}.$ Then the group $G:=G_{m,r}$ can be embedded in a division ring if and only if $G=Q_8$, a quaternion group of order $8$.
\end{lemma}
\begin{proof}
Since $\alpha \geq 2,$ we have that $4~|~m$, and hence, we have only one case to consider, namely, the case when $m=4.$ In this case, we get that $n=s=2$ and $r \equiv -1~(\textrm{mod}~4),$ and hence, $G$ can be $G_{4,r}=Q_8$. \\
\indent This completes the proof.
\end{proof}

Now, let $D$ be a division algebra of degree 2 over its center $K$ (i.e.\ $[D:K]=4$) where $K$ is an algebraic number field with $[K:\Q]=3$. Then it follows that $\textrm{dim}_{\Q}D = 12$ so that the order of every element of finite order of $D^{\times}$ is at most $18.$ If the group $G:=G_{m,r}$ is contained in $D^{\times},$ then $n~|~2$ (see \cite[\S7]{1}), and hence, we have either $n=1$ or $n=2.$ Also, since $G$ contains an element of order $m,$ we have $\varphi(m) ~|~ 6.$ The last preliminary result that we need is the following theorem.
\begin{theorem}[{\cite[Theorem 10]{1}}]\label{thm 17}
  Let $D$ be a division algebra of degree 2 over an algebraic number field $K$. \\
 (a) If $D$ contains a binary octahedral group $\mathfrak{O}^*,$ then $\sqrt{2}\in K.$ \\
 (b) If $D$ contains a binary icosahedral group $\mathfrak{I}^*,$ then $\sqrt{5}\in K.$
\end{theorem}
An immediate consequence of Theorem~\ref{thm 17} is the following.
\begin{corollary}\label{O*I* not poss cor}
Let $D$ be a division algebra of degree 2 over an algebraic number field $K$ with $[K:\Q]=3$. Then $D$ does not contain $\mathfrak{O}^*$ and $\mathfrak{I}^*.$
\end{corollary}

Summarizing, we have the following classification.
\begin{theorem}\label{thm 18}
  Let $D$ be a division algebra of degree 2 over an algebraic number field $K$ with $[K:\Q]=3$. A finite group $G$ (of even order\footnote{This assumption can be made based on the goal of this paper in the sense that the order of the automorphism group of a polarized abelian threefold over a finite field is even.}) can be embedded in $D^{\times}$ if and only if $G$ is one of the following groups: \\
  (1) $C_2, C_4, C_6, C_{14}, C_{18}$; \\
  (2) $Q_8, \textrm{Dic}_{12}, \textrm{Dic}_{28}, \textrm{Dic}_{36}$; \\
  (3) $\mathfrak{T}^*$, the binary tetrahedral group of order $24.$
\end{theorem}
\begin{proof}
  We refer the list of possible such groups to Theorem~\ref{thm 12} and Corollary~\ref{O*I* not poss cor}. Suppose that $G$ is cyclic. Then we can write $G=\langle f\rangle$ for some element $f$ of order $d$. Then according to the argument given before Theorem~\ref{thm 17}, we have $d \in \{2,4,6,14,18\}$. Hence, we obtain (1). If $G=G_{m,r}$ with $n=2$ and $m \in \{2,3,4,6,7,9,14,18\},$ then (2) follows from Lemmas~\ref{lem 15} and~\ref{lem 16}. (If $n=1,$ then $s=m$ so that $t=1.$ Now, the presentation of the group tells us that the group is cyclic of order $m$ in this case.) Now, if $G=\mathfrak{T}^* \times G_{m,r}$ is a general $T$-group, then the only possible such groups are $\mathfrak{T}^*$ and $\mathfrak{T}^* \times C_7.$ (If $n=2$ so that $G_{m,r}$ is not cyclic, then $|G_{m,r}|=2m,$ and hence, we have $\gcd(|G_{m,r}|,6) \ne 1.$) Also, clearly, both $\mathfrak{T}^*$ and $\mathfrak{T}^* \times C_7$ satisfy the condition of Theorem~\ref{thm 12}. On the other hand, if $G=\mathfrak{T}^* \times C_7,$ then, in view of \cite[Table 5]{Inneke}, we must have that $K=\Q(\zeta_7),$ which contradicts the fact that $[K:\Q]=3.$
\\
\indent Conversely, in view of Theorem~\ref{aimf of GL1} below, the two groups $\mathfrak{T}^*$ and $\textrm{Dic}_{12}$ are finite subgroups of $D^{\times}$ where $D=D_{2,\infty}$ and $D=D_{3,\infty}$, respectively. Let $K$ be a totally real cubic field. Then we can see that $\mathfrak{T}^*$ (resp.\ $\textrm{Dic}_{12}$) is a finite subgroup of $D^{\times}$ where $D=D_{2,\infty} \otimes_{\Q} K$ (resp.\ $D=D_{3,\infty} \otimes_{\Q} K)$, both of which are quaternion division algebras over $K.$ Also, by \cite[Theorem 6.1 (c)]{8}, the two groups $\textrm{Dic}_{28}$ and $\textrm{Dic}_{36}$ are finite subgroups of $D^{\times}$ where $D=D_{\zeta_{14}+\zeta_{14}^{-1}, \infty, 7}$ and $D=D_{\zeta_{18}+\zeta_{18}^{-1}, \infty, 3},$ respectively, where $D_{\zeta_m + \zeta_m^{-1}, \infty, p}$ denotes the quaternion division algebra over $\Q(\zeta_m+\zeta_m^{-1})$ that is ramified at the primes of $\Q(\zeta_m+\zeta_m^{-1})$ lying above $p$ and $\infty$. Then since $C_2 \leq C_4 \leq Q_8 \leq \mathfrak{T}^*$, $C_6 \leq \textrm{Dic}_{12},$ $C_{14} \leq \textrm{Dic}_{28},$ and $C_{18} \leq \textrm{Dic}_{36},$ the desired result follows. \\
\indent This completes the proof.
\end{proof}

\begin{remark}\label{Q8 not max rmk}
Let $D$ be as in Theorem~\ref{thm 18}. It turns out that the group $Q_8$ cannot be a maximal (up to isomorphism) finite subgroup of $D^{\times}$. Indeed, if $G=Q_8$ is a finite subgroup of $D^{\times},$ then by \cite[Theorem 9]{1}, we know that $D=D_{2,\infty} \otimes_{\mathbb{Q}} K$ for some number field $K$ with $[K:\mathbb{Q}] =3.$ Then it follows that $\mathfrak{T}^* \leq D^{\times},$ and since $Q_8 \leq \mathfrak{T}^*$, we can see that $Q_8$ cannot be a maximal finite subgroup of $D^{\times}.$
\end{remark}

\section{maximal finite subgroups of $GL_3(K)$ when $K$ is $\Q$ or a real quadratic field}\label{gl3 sec}
Throughout this section, either $K=\Q$ or $K$ is a real quadratic number field. Let $d >0$ be a square-free integer such that $K=\Q(\sqrt{d})$. (If $K=\Q,$ then we take $d=1$.) Note that $\mu_K =\{ \pm 1 \},$ where $\mu_K$ denotes the set of all roots of unity in $K.$ \\
\indent We begin with the following lemma, in which, we deal with finite subgroups of $GL_r(\C)$ consisting of diagonal matrices for some integer $r \geq 1$.

\begin{lemma}\label{Gal} Let $L/F$ be a finite Galois extension with $\Gamma = \textrm{Gal}(L/F)$, and let $H\subsetneq \bs{L^{\times}}^{r}\subsetneq GL_{r}(L)$ be a finite subgroup of diagonal matrices such that $\Tr{A}\in F$ for all $A\in H$. Then there exists a group action $*$ of $\Gamma$ on the set $\bm{1,2,\cdots,r}$ such that $^{\sigma}{x_{i}}=x_{\sigma* i}$ for all $\sigma\in \Gamma$ and $A=\Delta_{(x_1,\cdots,x_r)} \in H$ (where $\Delta_{(x_1,\cdots,x_r)}$ denotes the diagonal matrix with entries $x_1,\cdots,x_r$, here and in the sequel).
\end{lemma}

\begin{proof}
The inclusion map $\rho \colon H \hookrightarrow GL_r(\C)$ is a representation of $H$ whose associated character $\chi$ has values in $F$ by assumption. Since $H$ is abelian, $\rho$ is realizable over $F$ i.e.\ there is a matrix $P = (u_1 u_2 \cdots u_r) \in GL_r(L)$ such that $PHP^{-1} \subseteq GL_r(F).$ Hence $L^r$ decomposes uniquely into a direct sum $\bigoplus_{i \in I } W_i$ of maximal eigenspaces $W_i$ over $L$ of $PHP^{-1}$ for some index set $I.$ In other words, each $W_i$ is a maximal subspace of $L^r$ such that $S|_{W_i}$ is a scalar multiplication for any $S \in PHP^{-1}.$ \\
\indent Now, we define a group action $\circ$ of $\Gamma$ on the index set $I$ as follows: for any $\sigma \in \Gamma$ and $i \in I,$ note that $^{\sigma}W_i$ is a maximal eigenspace of $L^r$, and hence, there exists a unique $j \in I$ such that $^{\sigma}W_i = W_j.$ We let $\sigma \circ i = j.$ Then it is easy to see that $\circ$ is a group action of $\Gamma$ on $I.$ Moreover, we can choose bases $\mathcal{B}_i$ of $W_i$ for each $i \in I$ in such a way that $^{\sigma} \mathcal{B}_j =\mathcal{B}_{\sigma \circ j}$ for any $\sigma \in \Gamma$ and $j \in I.$ Then $\calB=\Union{i\in I}\calB_{i}$ is a basis of $L^{r}$ on which $\Gamma$ acts. Let $Q=(v_1 v_2 \cdots v_r) \in GL_r(L)$ be a square matrix consisting of column vectors in $\mathcal{B}$ such that $u_i$ and $v_i$ are in the same eigenspace for each $1 \leq i \leq r.$ Then we have $QAQ^{-1} = PAP^{-1}$ for all $A \in H.$ \\
\indent The action of $\Gamma$ on $\mathcal{B}$ induces an action $*$ of $\Gamma$ on the set $\{1,2,\cdots,r\}$ such that $^{\sigma} v_i = v_{\sigma * i}.$ Now, for any $A =\Delta_{(x_1,\cdots,x_r)}\in H,$ let $S=QAQ^{-1}=PAP^{-1} \in GL_r (F).$ Then it follows that
\begin{equation*}
x_{\sigma*i}v_{\sigma*i}=S{v_{\sigma*i}}=\conj{\sigma}{\bs{Sv_{i}}}=\conj{\sigma}{\bs{x_{i}v_{i}}}=\conj{\sigma}{x_{i}}v_{\sigma*i},
\end{equation*}
whence, $^{\sigma} x_i = x_{\sigma * i}$ for any $\sigma \in \Gamma$ and $i \in \{1,2,\cdots, r\}.$ \\
\indent This completes the proof.
\end{proof}
For our own purpose of this paper, we focus on the case when $r=3.$
\begin{lemma}\label{table_K}
	Let $H$ be a finite abelian subgroup of $GL_{3}(K)$ of exponent $N>1$. Then we have:\\
	(a) There is a finite subgroup $G$ of $GL_{3}(K)$ such that $G\cong  H \rtimes C_{2}$. Moreover, if $N > 2$, then there is a finite subgroup $G$ of $GL_{3}(K)$ that is not isomorphic to $H \times C_2$ (i.e.\ the conjugation action of such a $G$ on $H$ is not trivial). \\
	(b) If there is a finite subgroup $G$ of $GL_{3}(K)$ such that $G\cong  H \rtimes C_{3}$, then $H$ is similar to $\fkb=\subgp{\diag{-1,-1,1},\diag{1,-1,-1}}$, $\subgp{\fkb,-I_3}$, or a subgroup of $\mathfrak{a}= \subgp{-I_3,\diag{1,1,-1}}$. Moreover, $H$ is similar to $\fkb$ or $\subgp{\fkb,-I_3}$ if and only if $G\not \cong H\times C_{3}$, and $H$ is similar to a subgroup of $\fka$ if and only if $G\cong H\times C_{3}$.
\end{lemma}
\begin{proof}
We first observe that $H$ is similar to a subgroup of $GL_3(K)$ that consists of diagonal
matrices. Then some $P\in GL_{3}(K(\z{N}))$ diagonalizes $H$ i.e.\ $H':=P^{-1}HP\subseteq \bs{L^{\times}}^{3}$ where $L=K(\z{N})$. We find all possible $H'$ for each $N$.

If $N>2$, then there exists an $A\in H'$ of order $N$, and we may write $A=\diag{\z{N}^{a},\z{N}^{b},\z{N}^{c}}$ for some $a,b,c\in\Z$. Since $\abs{A}=N$, we get that $\gcd\bs{a,b,c,N}=1$ and $K(\z{N}^{a},\z{N}^{b},\z{N}^{c})=K(\z{N})$.\\
Since $\bl{K:\Q}=2$, $K\cap \Q(\z{N})$ equals either $K$ or $\Q$. If $K\cap \Q(\z{N})=K$, then $K\subseteq \Q(\z{N})$ and
$$
3\ge \bl{K(\z{N}^{a},\z{N}^{b},\z{N}^{c}):K}=\bl{K(\z{N}):K}=\bl{\Q(\z{N}):K}.
$$
Again, since $\bl{K:\Q}=2$, it follows that $\varphi(N)=\bl{\Q(\z{N}):\Q}\in \bm{2,4,6}$, and hence, we obtain
$$
N\in \bm{5,8,10,12}=\bm{3,4,5,6,7,8,9,10,12,14,18} \setminus \bm{3,4,6,7,9,14,18}
$$
because the unique quadratic number field in $\Q(\z{N})$ is not real for $N\in \bm{3,4,6,7,9,14,18}$. If $K\cap \Q(\z{N})=\Q$, then we have
$$
3\ge \bl{K(\z{N}):K}=\bl{K(\Q(\z{N})):K}=\bl{\Q(\z{N}):K\cap \Q(\z{N})}=\bl{\Q(\z{N}):\Q}=\varphi(N)
$$
so that $N\in\bm{3,4,6}$. To sum up, $\Gal(K(\z{N})/K)$ is of order $2$ and the non-trivial element of $\Gal(K(\z{N})/K)$ is the complex conjugation for any $K$. Because the diagonal entries are the solutions of the characteristic polynomial of $A$, which is defined over $K$, the group $\Gal(K(\z{N})/K)$ permutes the diagonal entries of $A$. Since $\abs{\Gal(K(\z{N})/K)}=2$, $\Gal(K(\z{N})/K)$ fixes a diagonal entry of $A$. We may assume that the last diagonal entry is fixed by $\Gal(K(\z{N})/K)$. Since $\abs{A}>2$, other diagonal entries of $A$ are not fixed. Thus we can say that $A=\diag{\z{N}^{a},\z{N}^{-a},\ep}$ for some $\ep\in \mu_{\Q}=\bm{\pm 1}$. Also, by Lemma~\ref{Gal}, any $B\in H'$ is of the form $\diag{\z{N}^{b},\z{N}^{-b},w}$ for some $b\in \Z$ and $w\in\bm{\pm1}$. Consequently, we see that $H'\leq \subgp{\diag{\z{N},\z{N}^{-1},1},\diag{1,1,-1}}$.

For $N=2$, let $r=\dim_{\F_{2}}H'>0$. If $r=1$, then $H'=\subgp{\diag{a,b,c}}$ for some $a,b,c\in \bm{\pm 1}$. At least two of $a, b,$ and $c$ are the same. We may assume that $a=b$, and then $H'$ is a subgroup of $\fka$. If $r=2$, then we consider the group homomorphism $\det \colon H'\to \bm{\pm1}$. Considering the orders of the domain and codomain of $\det$, we get that $|\ker\bs{\det}|\ge 2$. For a non-identity matrix $A\in \ker \bs{\det}$, we rearrange the diagonal entries of $A$ and may assume that $A=\diag{-1,-1,1}$. For $B\in H' \setminus \subgp{A}$, we may assume that $B=\diag{1,a,b}\ne I_3$ for some $a,b\in \bm{\pm 1}$ by replacing $B$ by $AB$ if the first diagonal entry of $B$ equals $-1$. Hence we can see that $H'$ is equal to $\subgp{\diag{-1,-1,1},\diag{1,-1,1}}$ or $\fka$, or $\fkb$. If $r=3$, then $H'=\subgp{\diag{-1,1,1},\diag{1,-1,1},\diag{1,1,-1}}=\subgp{\fkb,-I_3}$.

(a) It is easy to show that the matrix $\diag[\sigma]{1,1,-1}:=\psmatrix{0&1&0\\1&0&0\\0&0&-1}$ acts on any $H'$ and we have $G':=\subgp{H',\diag[\sigma]{1,1,-1}}\cong H'\rtimes C_{2}$. If $N>2$, the conjugation action of $\diag[\sigma]{1,1,-1}$ on $H'$ is not trivial.

(b) Assume that there is a finite group $G\lneq GL_{3}(K)$ such that $G\cong  H \rtimes C_{3}$. Then for some $X\in G':=P^{-1}GP$, one has that $\abs{X}=3$ and $X$ acts on $H'$ by conjugation. As before, $X$ permutes the diagonal entries of matrices of $H'$. There is a homomorphism $f \colon \subgp{X}\to Sym_{3}$ such that $XAX^{-1}=I_{3,\alpha}AI_{3,\alpha}^{-1}$ for any $A\in H'$ with $\alpha = f(X)$. Note that $\alpha$ is the identity $1\in Sym_{3}$ or a $3$-cycle. We claim that $N=2$. Indeed, if $N>2$, then the form of the matrices of $H'$ says that $\alpha=1$. Thus $G\cong G' :=H'\times \subgp{X}$ is abelian of exponent $\text{lcm}\bs{N,3}$. Since both $N$ and $\text{lcm}\bs{N,3}$ are contained in the set $\bm{3,4,6}$, we know that $N=\text{lcm}\bs{N,3}$ is equal to $3$ or $6$. Then since there is no abelian subgroup of $GL_3(K)$ of exponent $N$ and order $18=\abs{G}=\abs{G'}$, we conclude that $N=2$, and hence, the first assertion follows.

If $G\cong H\times C_{3}$, then $G$ is an abelian group of exponent $6$. Thus $G' :=P^{-1}GP$ is a subgroup of $\subgp{\diag{\z{6},\z{6}^{-1},1},\diag{1,1,-1}}$ and the $2$-torsion part $H'=P^{-1}HP$ of $G'$ is a subgroup of $\fka$. Conversely, if $H'$ is a subgroup of $\fka$, then the conjugation action by any matrix $X$ of order $3$ on $H'$ is trivial so that $G \cong H \times C_3$. Also, note that the form of the matrices of $H'$ says that $\alpha=1$.

Finally, it suffices to show that $H$ is not similar to $\subgp{\diag{-1,-1,1},\diag{1,-1,1}}$ to prove that $H$ is similar to $\fkb$ or $\subgp{\fkb,-I_3}$ if and only if $G\not \cong H\times C_{3}$. Indeed, if $H'=\subgp{\diag{-1,-1,1},\diag{1,-1,1}}$, then $\alpha=1$ and $G'$ is an abelian group of exponent $6$. However, $\subgp{\diag{-1,-1,1},\diag{1,-1,1}}$ is not a subgroup of $\fka$, which contradicts the above argument. \\
\indent This completes the proof.
\end{proof}

\begin{lemma}\label{dn_lemma}
Let $F$ be a number field and $n$ the order of the torsion part of $F^{\times}$ (i.e.\ we have $\mu_n = \mu_F$), and let $r>0$ be an integer that is relatively prime to $n.$ Then we have:\\
(a) If $G$ is a finite subgroup of $GL_r (F)$ containing $\mu_n \cdot I_r,$ then $G=\subgp{G_0, \mu_n \cdot I_r } \cong G_0 \times C_n$ where $G_0 = G \cap SL_r (F).$\\
(b) Assume that a finite subgroup $H$ of $GL_r(F)$ contains $\mu_n \cdot I_{r}.$ Then a group $G$ is a subgroup of $H$ if and only if $G_0$ is a subgroup of $H_0$, where $G_0 = G \cap SL_r(F)$ and $H_0 =H \cap SL_r(F)$.
\end{lemma}
\begin{proof}Note first that for any $z\in \mu_{n}$, there is an $x\in \mu_{n}$ such that $\det\bs{x \cdot I_r}=z$ because $\gcd\bs{r,n}=1$. \\
 \indent (a) Clearly, we have $\langle G_0, \mu_n \cdot I_r \rangle \subseteq G$ by our assumption. To prove the reverse inclusion, let $A \in G.$ Then we have $\textrm{det}(A) \in \mu_n,$ and hence, there is an $x \in \mu_n$ such that $x^{-1} A \in G_0$ by the above observation. Thus we can see that $A=(x \cdot I_r ) \cdot (x^{-1} A) \in \langle G_0, \mu_n \cdot I_r \rangle,$ and hence, we have $A \subseteq \langle G_0, \mu_n \cdot I_r \rangle$. Now, since $\gcd(r,n)=1,$ we have $G_0 \cap (\mu_n \cdot I_r) = \{I_r \}.$ Moreover, any element of $\mu_n \cdot I_r$ commutes with any element of $G_0$. Therefore, we can conclude that $G \cong G_0 \times C_n.$ \\
\indent (b) It is clear that if $G \leq H,$ then $G_0 \leq H_0$ by definition. To prove the converse, we note that $G \leq \langle G_0, \mu_n \cdot I_r \rangle$ and $H=\langle H_0, \mu_n \cdot I_r \rangle$ by part (a). Since $G_0 \leq H_0,$ it follows that $G \leq H$. \\
\indent This completes the proof.
\end{proof}

As an immediate consequence of the above lemma, we get the following corollary.
\begin{corollary}\label{dn_corollary}
Let $F$, $n$, $r$, and $G$ be as in Lemma~\ref{dn_lemma}. If $G$ is a maximal finite subgroup of $GL_{r}(F)$, then  $G=\subgp{G_{0},\mu_{n} \cdot I_r}\cong G_{0}\times C_{n}$ where $G_{0}$ is a maximal finite subgroup of $SL_{r}(F)$.
\end{corollary}
\begin{remark}\label{reduction rmk}
In light of Corollary~\ref{dn_corollary}, we can reduce the problem of finding all maximal finite subgroups of $GL_{3}(K)$ to that of finding all maximal finite subgroups of $SL_{3}(K)$ because we have $\mu_{K}=\mu_{2}$$=\{\pm 1\}$. In other words, there is a 1-1 correspondence between maximal finite subgroups of $GL_{3}(K)$ and maximal finite subgroups of $SL_{3}(K)$ given by $G \mapsto G_0 = G \cap SL_3(K)$. The inverse of the map is given by $G_0 \mapsto G= \subgp{G_0, -I_3}$.
\end{remark}

In view of Remark~\ref{reduction rmk}, we classify all maximal finite subgroups $G_0$ of $SL_3(K)$ for some $K$ in the following subsequent lemmas.
\begin{lemma}\label{(E),(G) lem}
Let $G_0$ be a finite subgroup of $\textrm{SL}_3(\C)$ of the type (E), (F), (G), (I), (J), (K), or (L) in \cite[Definition 17]{JCS}. Then there is no $K$ such that $G_0 \leq \textrm{SL}_3(K).$
\end{lemma}
\begin{proof}
This follows immediately from \cite[$\S$5.3]{RM} by observing that there is no $K$ such that $|G_0|$ divides $S(3, K)$ for all of these types.
\end{proof}


\begin{lemma}\label{(A),(C) lem}
Let $G_0$ be a finite subgroup of $\textrm{SL}_3(\C)$ of the type (A) or (C) or (D) in \cite[Definition 17]{JCS}. If $G_0$ is a maximal (up to isomorphism) finite subgroup of $\textrm{SL}_3(K)$ for some $K,$ then $G_0 = \textrm{Sym}_4$ and $K$ is arbitrary.
\end{lemma}
\begin{proof}
Let $H=\{A \in G_0~|~ A ~\textrm{is diagonal}\} \leq G_0.$ Then we need to consider the following three cases: \\
(i) If $G_0$ is of type (A), then we have $G_0 = H.$ \\
(ii) If $G_0$ is of type (C), then we have $G_0 =\langle H, T \rangle$ where $T=\psmatrix{0&1&0\\0&0&1\\1&0&0}.$ \\
(iii) If $G_0$ is of type (D), then we have $G_0 =\langle H, T, Q \rangle$ where $Q=\psmatrix{0&1&0\\1&0&0\\0&0&-1}.$ \\
By Lemma~\ref{table_K}, it follows that such a unique maximal $G_0$ is equal to $G_0 = \langle \fkb, T, Q \rangle \cong \textrm{Sym}_4$ (among the groups of type (A), (C), or (D)). Now, we show that $G_0= \langle \fkb, T, Q \rangle$ is a maximal finite subgroup of $SL_3(K).$ Indeed, suppose on the contrary that there is a finite subgroup $G_0^{\prime}$ of $SL_3(K)$ such that $G_0 \lneq G_0^{\prime}.$ Then by the above observation, $G_0^{\prime}$ cannot be of type (A), (C), or (D). Also, $G_0^{\prime}$ cannot be of type (E), (F), (G), (I), (J), (K), or (L) by Lemma~\ref{(E),(G) lem}. Hence it suffices to consider the following two cases: \\
\indent (Case 1): $G_0^{\prime}$ is of type (B). In this case, we can show that $G_0^{\prime}$ is a dihedral group (see Lemma~\ref{(D) lem} below). Then since $G_0 \leq G_0^{\prime},$ $G_0$ must be either cyclic or a dihedral group, which is absurd. \\
\indent (Case 2): $G_0^{\prime}$ is of type (H). In this case, we have $|G_0^{\prime}|=60$ and this contradicts the fact that $|G_0|$ must divide $|G_0^{\prime}|.$ \\
Therefore, from (Case 1) and (Case 2), we can conclude that $G_0$ is a maximal finite subgroup of $SL_3(K),$ as desired.\\
\indent This completes the proof.
\end{proof}

\begin{lemma}\label{(D) lem}
Let $G_0$ be a finite subgroup of $\textrm{SL}_3(\C)$ of the type (B) in \cite[Definition 17]{JCS}. If $G_0$ is a maximal (up to isomorphism) finite subgroup of $\textrm{SL}_3(K)$ for some $K,$ then $G_0 = D_{n} $ for some $n$, and $\zeta_n + \zeta_n^{-1} \in K$.
\end{lemma}
\begin{proof}
Note that $G_0 = \{ \widetilde{A} ~|~A \in H \}$ for some finite subgroup $H$ of $GL_2(\C).$ Let $H_0 = H \cap SL_2(\C).$ Then by \cite[$\S$2]{JCS}, we have
\begin{equation*}
H_0 \in \{\langle \Delta_n \rangle,~ \langle \Delta_n, R \rangle, \langle \Delta_4, R, C \rangle, \langle \Delta_8, R, C \rangle, \langle \Delta_{10}, R, D \rangle : n~\textrm{is even} \}
\end{equation*}
where $\Delta_n = \Delta_{(\zeta_n^{-1},\zeta_n)},$ $R=\psmatrix{0&1\\-1&0}$, $C=\frac{1}{2}\psmatrix{-1-i&\phantom{..}& 1-i\\\phantom{..}&\phantom{..}&\phantom{..}\\-1-i&\phantom{..}&-1+i}$, and $D=\frac{1}{\sqrt{5}}\psmatrix{-\z{5}+\z{5}^{-1} & \z{5}^{2}-\z{5}^{-2}\\ \z{5}^{2}-\z{5}^{-2}& \z{5}-\z{5}^{-1}}$. We claim further that $H_0 = \langle \Delta_n \rangle$ for an even integer $n.$ Indeed, suppose on the contrary that $H_0$ is similar to one of the groups $\langle \Delta_n, R \rangle ~(n: \textrm{even})$ or $\langle \Delta_n, R, C \rangle$ for $n \in \{4,8\}$ or $\langle \Delta_{10}, R, D \rangle$. Then it follows that $G$ contains a subgroup $\langle \widetilde{\Delta_n }, \widetilde{R} \rangle$, that is isomorphic to the dicyclic group $\textrm{Dic}_{2n},$ which is absurd, because of the well-known classification of finite subgroups of $GL_3(\R)$ (see \cite[Example 4 in $\S13$]{Sha(2005)}). Hence we can see that $H_0 = \langle \Delta_n \rangle$ for some even integer $n.$ \\
\indent Now, let $\widetilde{A} \in \widetilde{H}=G_0.$ In view of the proof of Lemma~\ref{table_K}, we can see that $\widetilde{A}$ commutes with $\widetilde{\Delta_n }$ if and only if $A \in H_0.$ Let $B \in H \setminus H_0.$ Then $B \Delta_n B^{-1} \in H_0$ and it is a diagonal matrix. Since the conjugation by the matrix $B$ permutes the diagonal entries of $\Delta_n$, we have $B \Delta_n B^{-1} = \Delta_n^{-1}.$ Hence it follows that $B=\psmatrix{0&x\\y&0}$ for some $x,y$, and $\widetilde{B}^2 = \Delta_{(xy, xy, 1/(xy)^2)} \in G_0.$ Then since $\langle \widetilde{B}^2 \rangle $ is an abelian finite subgroup of $GL_3(K),$ the proof of Lemma~\ref{table_K} shows that $xy=\pm 1.$ Since $B \not \in H_0,$ we can see that $B= \psmatrix{0&x \\ x^{-1} & 0}.$ Now, for any $B^{\prime} \in H$ which does not commute with $\Delta_n,$ we can see that $BB^{\prime}$ commutes with $\Delta_n$ and $BB^{\prime} \in H_0.$ Hence it follows that $G_0 \in \left\{\langle \widetilde{\Delta_n} \rangle, \langle \widetilde{\Delta_n}, \widetilde{B} \rangle \right\},$ and then, since $G_0$ is assumed to be maximal, we can conclude that $G_0 =  \langle \widetilde{\Delta_n}, \widetilde{B} \rangle \cong D_{n}.$ The last assertion follows from the observation that the characteristic polynomial of $\widetilde{\Delta_n} \in G_0$ is defined over $K$ if and only if the characteristic polynomial of $\Delta_n \in H_0$ is defined over $K.$ \\
\indent This completes the proof.
\end{proof}

\begin{lemma}\label{(F) lem}
Let $G_0$ be a finite subgroup of $\textrm{SL}_3(\C)$ of the type (H) in \cite[Definition 17]{JCS}. If $G_0$ is a maximal (up to isomorphism) finite subgroup of $\textrm{SL}_3(K)$ for some $K,$ then $G_0 = \textrm{Alt}_5$ and $K=\Q(\sqrt{5})$.
\end{lemma}
\begin{proof}
This follows from a similar argument as in the proof of Lemma~\ref{(A),(C) lem}.
\end{proof}



Consequently, we can obtain all maximal finite subgroups of $GL_3(K)$ for some $K$ as follows.
\begin{corollary}\label{max cor gl3}
Let $G$ be a maximal (up to isomorphism) finite subgroup of $\textrm{GL}_3(K)$ for some $K.$ Then $G \in \{C_2 \wr \textrm{Sym}_3, D_{n} \times C_2 ~(\textrm{some}~n), \textrm{Alt}_5 \times C_2 \}.$
\end{corollary}
\begin{proof}
This follows from Remark~\ref{reduction rmk} and Lemmas~\ref{(E),(G) lem},~\ref{(A),(C) lem},~\ref{(D) lem},~\ref{(F) lem}, together with the fact that $\textrm{Sym}_4 \times C_2 \cong C_2 \wr \textrm{Sym}_3.$
\end{proof}
The next lemma says that the group $D_n \times C_{2}$ (for $n \geq 4$ even) is isomorphic to a subgroup of $GL_3(\Q(\zeta_n + \zeta_n^{-1}))$, which is not irreducible.
\begin{lemma}\label{not irre lem}
Let $n>2$ be an even integer and let $\rho$ be a $3$-dimensional representation of the group $D_n \times C_{2}$. Then $\rho$ realizes over $\Q(\zeta_n + \zeta_n^{-1})$ and it is not irreducible.
\end{lemma}
\begin{proof}
Note that it suffices to show that all irreducible representations of $D_n \times C_{2}$ have dimension at most $2$ and realize over $\Q(\zeta_n + \zeta_n^{-1})$. Then any subgroup of $GL_{3}(\Q(\zeta_n + \zeta_n^{-1}))$ which is isomorphic to $D_n \times C_{2}$ is not irreducible. Indeed, we recall that the group $D_n$ has the presentation $D_n = \langle \alpha, \beta~|~\alpha^n = \beta^2= 1, \beta \alpha \beta^{-1}=\alpha^{-1} \rangle$, and $D_n$ has four $1$-dimensional irreducible representations and $(n/2-1)$ number of $2$-dimensional irreducible representations. Any $1$-dimensional irreducible character is given by $\chi_{x,y}\colon \alpha^{m}\mapsto x^{m},\alpha^{m}\beta \mapsto x^{m}y$ for some $\bs{x,y}\in\bm{\pm1}^{2}$. Each $2$-dimensional irreducible character is given by $\chi_{u}\colon \alpha^{m}\mapsto \z{n}^{m u}+\z{n}^{-m u},\alpha^{m}\beta \mapsto 0$ for each $1\le u\le n/2-1$. Since $D_n$ can be embedded into $GL_2(\Q(\zeta_n + \zeta_n^{-1}))$ and $SL_3(\Q(\zeta_n + \zeta_n^{-1})),$ $\chi_{u}$ realizes over $\Q(\zeta_n + \zeta_n^{-1})$ for all $1\le u\le n/2-1$. Also, any irreducible representation of $C_{2}$ realizes over $\Q(\zeta_n + \zeta_n^{-1})$. Hence, we can see that any irreducible representation of $D_n \times C_{2}$ has dimension at most $2$ and realizes over $\Q(\zeta_n + \zeta_n^{-1})$. \\
\indent This completes the proof.
\end{proof}

\begin{remark}\label{QG_dih3}
We compute the rational group algebra $\Q G$ of $G=\subgp{\wtilde{\D_{n}},\diag{1,1,-1}} \cong D_{n} \times C_2$ where $\D_n=\left\langle \psmatrix{0&1\\1&\zeta_n + \zeta_n^{-1}}, \psmatrix{1&0\\ \zeta_n + \zeta_n^{-1} & -1} \right\rangle$. More precisely, we show that
$$\Q G=M_{2}(F)\bigoplus M_{1}(\Q)$$
where $F=\Q(\z{n}+\z{n}^{-1})$. Indeed, clearly, we have $\Q G \subset M_{2}(F)\bigoplus M_{1}(\Q)$. Now, since $\diag{1,1,0}\in \Q\subgp{\diag{1,1,-1}}$ and $\Q \D_{n}=M_{2}(F)$,
we have
\begin{equation*}
M_{2}(F)\bigoplus 0=\Q \wtilde{\D_{n}}\diag{1,1,0}\subseteq \Q G,
\end{equation*}
and since $\diag{0,0,1}\in \Q\subgp{\diag{1,1,-1}}$, we also have \[
0\bigoplus M_{1}(\Q)\subseteq \Q G.
\] Hence it follows from dimension counting that $\Q G = M_2(F) \oplus M_1(\Q),$ as desired.
\end{remark}

In summary, we obtain the following useful result.
\begin{lemma}\label{imag lem5}
Let $G$ be an irreducible maximal (up to isomorphism) finite subgroup of $GL_{3}(K)$ for some $K=\Q(\sqrt{d})$ where $d>0$ is a square-free integer. Then $G$ is one of the following groups:
\begin{center}
	\begin{tabular}{ccc}
		\hline
		& $G$  & $d$ \\
		\hline
		$\sharp 1$ & $C_{2}\wr \textrm{Sym}_{3}$ & any $d$ \\
		\hline
		$\sharp 2$ & $\textrm{Alt}_{5}\times C_{2}$  & $5$ \\
		\hline
	\end{tabular}
	\vskip 4pt
	\textnormal{Table 1}
\end{center}
\end{lemma}
\begin{proof}
In view of Corollary~\ref{max cor gl3} and Lemma~\ref{not irre lem}, we only need to consider those two groups in Table 1 above.
The irreducibility of the two groups can be derived directly from the Schur orthogonality relations. To show that $\textrm{Alt}_5 \times C_2 \cong$$\subgp{F_{60},-I_{3}}$ is similar to a subgroup of $GL_{3}(\Q(\sqrt{5}))$, we show that the Schur index of $\subgp{F_{60},-I_{3}}$ equals $1$. (Here, $F_{60}$ denotes the group generated by $\psmatrix{1&0&0\\ 0 & \zeta_5^{-1} & 0 \\ 0 & 0 & \zeta_5}, \psmatrix{-1&0&0\\ 0 & 0 & -1 \\ 0 & -1 & 0}$, and $\frac{1}{\sqrt{5}} \psmatrix{1&1&1\\ 2 & \zeta_5^2 +\zeta_5^{-2} & \zeta_5 + \zeta_5^{-1} \\ 2 & \zeta_5 + \zeta_5^{-1} & \zeta_5^2 +\zeta_5^{-2}}.$) To this aim, let $G=\subgp{F_{60},-I_{3}}$ and consider the inclusion representation $\rho \colon G\hookrightarrow GL_3(\C)$. Let $m$ be the Schur index of $\rho$. Then it is well known that we have $m^{2}\mid |G|= 120=2^{3}\cdot 3^{1}\cdot 5^{1}$ and $m\mid \dim \rho = 3$, and hence, it follows that $m=1$. \\
\indent This completes the proof.
\end{proof}
\begin{remark}\label{GQ_3}
We compute the rational group algebra $\Q G$ of the two groups $G$ in Table 1. \\
(i) For $C_2 \wr \textrm{Sym}_3 \cong G= \subgp{\diag{-1,1,1},I_{3,\alpha}:\alpha\in \textrm{Sym}_{3}}$, we have $\Q G=M_{3}(\Q)$. Indeed, clearly, we have $\Q G \subseteq M_3(\Q).$ Also, it is easy to see that $\bm{d_{i}T^{j}:1\le i,j\le 3}$ are linearly independent where $d_{1}=\diag{-1,1,1}$, $d_{2}=\diag{1,-1,1}$, $d_{3}=\diag{1,1,-1}$, and $T=\psmatrix{0&1&0\\0&0&1\\1&0&0}$. Then it follows that $9\le \dim_{\Q}\bs{\Q G}\le \dim_{\Q}\bs{M_{3}(\Q)}=9$, and hence, we get $\Q G=M_{3}(\Q)$. \\
(ii) Note that there is a $P\in GL_{3}(\C)$ such that $\textrm{Alt}_5 \times C_2 \cong G=P \cdot \subgp{F_{60},-I_3} \cdot P^{-1} \subsetneq GL_{3}(\Q(\sqrt{5}))$ according to the argument in the proof of Lemma~\ref{imag lem5}. Then by a similar argument as in (i), together with the observation that we have \begin{align*}
\Q \subgp{F_{60},-I_3}&\supseteq \bm{\psmatrix{0&0&0\\0&\alpha&0\\0&0&\wbar{\alpha}\\}:\alpha\in \Q(\z{5})} \bigoplus \bm{\psmatrix{0&0&0\\0&0&\alpha\\0&\wbar{\alpha}&0\\}:\alpha\in \Q(\z{5})}\\
&\bigoplus \bm{\psmatrix{0&0&0\\\alpha&0&0\\\wbar{\alpha}&0&0\\}:\alpha\in \Q(\z{5})}
\bigoplus \bm{\psmatrix{0&\alpha&\wbar{\alpha}\\0&0&0\\0&0&0\\}:\alpha\in \Q(\z{5})}\\
&\bigoplus \bm{\psmatrix{\beta&0&0\\0&0&0\\0&0&0\\}:\beta\in K=\Q(\sqrt{5})},
\end{align*}
we can also see that $\Q G = M_3(\Q(\sqrt{5})).$
\end{remark}

\section{Quaternionic matrix representations}\label{quat mat rep}
In this section, we introduce some facts about finite quaternionic matrix groups that will be used later, following a paper of Nebe~\cite{8}. \\

Throughout this section, let $\mathcal{D}$ be a definite quaternion algebra over a totally real number field $K$ and let $V=\mathcal{D}^{1 \times n}$ be a right module over $M_n(\mathcal{D}).$ Endomorphisms of $V$ are given by left multiplication by elements of $\mathcal{D}.$ For computations, it is also convenient to let the endomorphisms act from the right. Then we have $\textrm{End}_{M_n(\mathcal{D})}(V) \cong \mathcal{D}^{\textrm{op}}$, which we identify with $\mathcal{D}$ since $\mathcal{D}$ is a quaternion algebra. \\

Now, we start with the following definition.
\begin{definition}[{\cite[Definition 2.1]{8}}]\label{Nebe def}
Let $G$ be a finite group and $\Delta \colon G \rightarrow GL_n (\mathcal{D})$ be a representation of $G$. \\
(a) Let $L$ be a subring of $K$. Then the \emph{enveloping $L$-algebra} $\overline{L \Delta(G)}$ is defined as
\begin{equation*}
\overline{L \Delta(G)} = \left\{\sum_{x \in \Delta(G)} l_x \cdot x~|~l_x \in L \right\} \subseteq M_n (\mathcal{D}).
\end{equation*}
(b) $\Delta$ is called \emph{absolutely irreducible} if the enveloping $\Q$-algebra $\overline{\Delta(G)}:=\overline{\Q \Delta(G)}$ of $\Delta(G)$ equals $M_n(\mathcal{D}).$ \\
(c) $\Delta$ is \emph{irreducible} if the commuting algebra $C_{M_n(\mathcal{D})}(\Delta(G))$ is a division algebra. \\
(d) A subgroup $G \leq GL_n(\mathcal{D})$ is called \emph{irreducible} (resp.\ \emph{absolutely irreducible}) if the natural representation $\textrm{id} \colon G \rightarrow GL_n(\mathcal{D})$ is irreducible (resp.\ absolutely irreducible).
\end{definition}

\begin{remark}[{\cite[page 110]{8}}]\label{irred max}
(a) The irreducible maximal finite subgroups $G$ of $GL_n(\mathcal{D})$ are absolutely irreducible in their enveloping $\Q$-algebras $\overline{G}$ where $\overline{G} \cong M_m(\mathcal{D}^{\prime})$ for some integer $m \geq 1$ and division algebra $\mathcal{D}^{\prime}$ with $m^2 \cdot \textrm{dim}_{\Q} \mathcal{D}^{\prime} $ dividing $n^2 \cdot \textrm{dim}_{\Q} \mathcal{D}.$ \\
(b) The reducible maximal finite subgroups of $GL_n(\mathcal{D})$ can be built up from the irreducible maximal finite subgroups of $GL_l (\mathcal{D})$ for $l < m.$
\end{remark}

For Lemmas~\ref{irred max lem},~\ref{imag lem},~\ref{irre max real mat},~\ref{irre max Q mat}, and~\ref{div alg lem} below, we assume further that $\mathcal{D}=D_{p,\infty}$ for some prime number $p.$ Then in view of Remark~\ref{irred max} and the double centralizer theorem, we have the following result.
\begin{lemma}\label{irred max lem}
  Let $G$ be an irreducible, but not absolutely irreducible, maximal finite subgroup of $GL_3(\mathcal{D})$. Then $\overline{G}$ is one of the followings:
  \begin{center}
  \begin{tabular}{cccc}
\hline
$$ & $\overline{G}$ & $\textrm{dim}_{\Q} \overline{G}$ & $C_{M_3(\mathcal{D})}(\overline{G})$ \\
\hline
$\sharp 1$ & $M_3(K^{\prime})$ ~$\textrm{where~$K^{\prime}$~is a quadratic field}$ & $18$ & $K^{\prime}$ \\
\hline
$\sharp 2$ & $\mathcal{D}^{\prime},$ a quaternion division algebra over a cubic field $K^{\prime}$ & $12$ & $K^{\prime}$ \\
\hline
$\sharp 3$ & $M_3(\Q)$ & $9$ & $\mathcal{D}$ \\
\hline
$\sharp 4$ & $K^{\prime}$, a sextic field & $6$ & $K^{\prime}$ \\
\hline
\end{tabular}
\vskip 4pt
\textnormal{Table 2}
\end{center}
\end{lemma}

We examine each case of $\sharp1 \sim \sharp3$ of Lemma~\ref{irred max lem} in the subsequent lemmas. We first consider the cases $\sharp1$ and $\sharp3$.
\begin{example}\label{GL(2,3)}
  Let $G=GL_3(\F_2) \times C_2.$ Then it is known in \cite{BGL(1999)} that $G$ is a maximal finite subgroup of $GL_3(\Q(\sqrt{-7})).$ Now, since $13$ is inert in $\Q(\sqrt{-7})$ so that $\Q(\sqrt{-7}) \subset D_{13,\infty},$ we have $G \leq GL_3(D_{13,\infty}).$  We claim that $G$ is a maximal finite subgroup of $GL_3(D_{13,\infty})$ (up to isomorphism). Indeed, suppose on the contrary that there is a finite subgroup $H$ of $GL_{3}(D_{13,\infty})$ such that $G$ is (isomorphic to) a proper subgroup of $H$. Since $G$ is an absolutely irreducible maximal finite subgroup of $GL_3(\Q(\sqrt{-7})),$ it follows that $\overline{H}=M_3(D_{13,\infty})$ by dimension counting over $\Q$. This means that $H$ is an absolutely irreducible finite subgroup of $GL_3(D_{13,\infty}),$ and this contradicts Theorems~\ref{prim aimf of GL3} and~\ref{imprim aimf GL2} below. Hence, $G$ is a maximal finite subgroup of $GL_3(D_{13,\infty}).$ Also, by definition, $G$ is irreducible because we have $C_{M_3(D_{13,\infty})}(\overline{G})=\Q(\sqrt{-7}),$ which is a division algebra. Similar arguments apply for cases $\sharp 2 \sim \sharp 4$ of Table 3 below, with $p=11.$
\end{example}

By a similar argument as in Example~\ref{GL(2,3)}, together with the use of \cite[Theorem 6.1]{BGL(1999)}, we can obtain the following lemma.
\begin{lemma}\label{imag lem}
  Let $G$ be an irreducible maximal (up to isomorphism) finite subgroup of $GL_3(\mathcal{D})$ with $\mathcal{D}=D_{p,\infty}$ for some prime $p$, whose enveloping $\Q$-algebra is $M_3(K^{\prime})$ for some imaginary quadratic field $K^{\prime}$. Then $G$ is one of the following groups:
 \begin{center}
    \begin{tabular}{cc}
\hline
$$ & $G$   \\
\hline
$\sharp 1$ & $GL_3(\F_2) \times C_2$ \\ 
\hline
$\sharp 2$ & $C_4^3 \rtimes \textrm{Sym}_3 \cong C_4 \wr \textrm{Sym}_3$ \\ 
\hline
$\sharp 3$ & $C_6^3 \rtimes \textrm{Sym}_3 \cong C_6 \wr \textrm{Sym}_3$  \\
\hline
$\sharp 4$ & $[([(C_3 \times C_6) \rtimes \textrm{Sym}_3 ] \cdot C_2) \cdot C_2 ]\cdot C_3 \cong (\textit{He}_3 \times \langle -\textbf{1}_3 \rangle ) \rtimes SL_2(\F_3)$    \\
\hline
\end{tabular}
\vskip 4pt
\textnormal{Table 3}
\end{center}
\end{lemma}

If $\overline{G} = M_3 (K^{\prime})$ for some totally real field with $[K^{\prime}:\Q] \leq 2,$ then in view of Lemma~\ref{imag lem5}, we have the following two results.
\begin{lemma}\label{irre max real mat}
There is no irreducible maximal (up to isomorphism) finite subgroup $G$ of $GL_3(\mathcal{D})$ with $\mathcal{D}=D_{p,\infty}$ for some prime $p$, whose enveloping $\Q$-algebra is $M_3(K^{\prime})$ for some real quadratic field $K^{\prime}$.

\end{lemma}
\begin{proof}
By Lemma~\ref{imag lem5} and Remark~\ref{GQ_3}, we only need to consider the group $G:=\textrm{Alt}_5 \times C_2$. Suppose on the contrary that $G$ is an irreducible maximal finite subgroup of $GL_3(\mathcal{D})$ with $\mathcal{D}=D_{p,\infty}$ for some prime $p.$ Recall that we have $\overline{G}=M_3(\Q(\sqrt{5})).$ Then since we have
\begin{equation*}
M_3(\Q) \subseteq \overline{G} =M_3(\Q(\sqrt{5})) \subseteq M_3(\mathcal{D}),
\end{equation*}
it follows that
\begin{equation*}
C_{M_3(\mathcal{D})}(M_3(\mathcal{D})) \subseteq C_{M_3(\mathcal{D})}(M_3(\Q(\sqrt{5}))) \subseteq C_{M_3(\mathcal{D})}(M_3(\Q)),
\end{equation*}
or equivalently,
\begin{equation*}
\Q \subseteq \Q(\sqrt{5}) \subseteq \mathcal{D}=D_{p,\infty}.
\end{equation*}
But then, since the infinite place of $\Q$ splits completely in $\Q(\sqrt{5}),$ the last inclusion is absurd. \\
\indent This completes the proof.
\end{proof}






\begin{lemma}\label{irre max Q mat}
Let $G$ be an irreducible maximal (up to isomorphism) finite subgroup of $GL_3(\mathcal{D})$ with $\mathcal{D}=D_{p,\infty}$ for some prime $p,$ whose enveloping $\Q$-algebra is $M_3(\Q)$. Then $G=C_2 \wr \textrm{Sym}_3$.  

\end{lemma}
\begin{proof}
This follows from Lemma~\ref{imag lem5} and Remark~\ref{GQ_3}. In fact, $G$ is an irreducible maximal finite subgroup of $GL_3 (\mathcal{D})$ with $\mathcal{D}=D_{p,\infty}$ for $p=109$.
\end{proof}


Now, we move to the case $\sharp2$ of Lemma~\ref{irred max lem}.
\begin{lemma}\label{div alg lem}
  Let $G$ be an irreducible maximal (up to isomorphism) finite subgroup of $GL_3(\mathcal{D})$ with $\mathcal{D}=D_{p,\infty}$ for some prime $p$, whose enveloping $\Q$-algebra is $M_1(\mathcal{D}^{\prime})$ for a quaternion division algebra $\mathcal{D}^{\prime}$ over a cubic field $K^{\prime}.$ Then $G=\textrm{Dic}_{28}$ with $K^{\prime}=\Q(\zeta_{14}+\zeta_{14}^{-1}).$
\end{lemma}
\begin{proof}
In light of Theorem~\ref{thm 18}, together with the computation of $\overline{G}$, we can see that all the possible candidates for such maximal finite subgroups are $\textrm{Dic}_{28}$ and $\textrm{Dic}_{36}.$ 
Now, we consider those two groups one by one. \\
\indent (1) Take $G:=\textrm{Dic}_{28}.$ Let $K=\Q(\zeta_{14}+\zeta_{14}^{-1}).$ By the first theorem of \cite[page 407]{LS(2012)}, we have $K \subseteq M_3(\Q)$, which, in turn, implies that $K \subseteq M_3(D_{7,\infty}).$ Then it follows that $D_{7,\infty} \otimes_{\Q} K \subseteq D_{7,\infty} \otimes_{\Q} M_3(\Q) \cong M_3 (D_{7,\infty})$. Now, we recall that $\overline{G}=D_{\zeta_{14}+\zeta_{14}^{-1}, \infty, 7} \cong D_{7,\infty} \otimes_{\Q} K$ by \cite[Theorem 6.1]{8}. Also, note that $\textrm{Dic}_{28}$ is not isomorphic to a subgroup of $\pm L_2(7).2$, the latter being a primitive absolutely irreducible maximal finite subgroup of $GL_3(D_{7,\infty})$ (see Theorem~\ref{prim aimf of GL3} below), and hence, $\textrm{Dic}_{28}$ is maximal. \\
\indent (2) For the group $G:=\textrm{Dic}_{36},$ we take $K:=\Q(\zeta_{18}+\zeta_{18}^{-1})$, and then, we can proceed as in the proof of (1) with the choices of $D_{3,\infty} \otimes_{\Q} K \cong D_{\zeta_{18}+\zeta_{18}^{-1}, \infty, 3}$ and $M_3(D_{3,\infty}).$ Here, note that we have $\textrm{Dic}_{36} \leq \textrm{Dic}_{12} \wr \textrm{Sym}_3$, the latter being an imprimitive absolutely irreducible maximal finite subgroup of $GL_3(D_{3,\infty})$ (see Theorem~\ref{imprim aimf GL2} below), and hence, $\textrm{Dic}_{36}$ is not maximal.\\
\indent This completes the proof.
\end{proof}

Now, for the rest of this section, $\mathcal{D}$ will denote a totally definite quaternion algebra over a totally real number field $K$, unless otherwise specified. As in the case of $GL_n (\Q),$ the notion of primitivity gives an important reduction in the determination of the maximal finite subgroups of $GL_n(\mathcal{D}).$

\begin{definition}[{\cite[Definition 2.2]{8}}]\label{prim def}
  Let $G $ be an irreducible finite subgroup of $GL_n(\mathcal{D})$. Consider $V:=\mathcal{D}^{1 \times n}$ as a $\mathcal{D}$-$G$-bimodule. Then $G$ is called \emph{imprimitive} if there is a decomposition $V=V_1 \oplus \cdots \oplus V_s$ $(s>1)$ of $V$ as a direct sum of nontrivial $\mathcal{D}$-left modules such that $G$ permutes the $V_i$ (i.e.\ for all $x \in G$ and for all $1 \leq i \leq s,$ there is a $j \in \{1,\cdots,s\}$ such that $V_i x \subseteq V_j$). If $G$ is not imprimitive, then $G$ is called \emph{primitive}.
\end{definition}

The following remark gives a way to produce the imprimitive absolutely irreducible maximal finite subgroups of $GL_n(\mathcal{D})$ from the primitive absolutely irreducible maximal finite subgroups of smaller degree.

\begin{remark}[{\cite[page 110]{8}}]\label{impri and prim}
  Let $G $ be an imprimitive absolutely irreducible maximal finite subgroup of $GL_n(\mathcal{D})$. Since $G$ is maximal and finite, it follows that $G$ is a wreath product of primitive absolutely irreducible maximal finite subgroups of $GL_d(\mathcal{D})$ with the symmetric group $\textrm{Sym}_{\frac{n}{d}}$ of degree $\frac{n}{d}$ for divisors $d$ of $n.$ In particular, every imprimitive absolutely irreducible maximal finite subgroup of $GL_3(\mathcal{D})$ is a wreath product of primitive absolutely irreducible maximal finite subgroups of $GL_1(\mathcal{D})=\mathcal{D}^{\times}$ with the symmetric group $\textrm{Sym}_3.$
\end{remark}

Now, we introduce two results about absolutely irreducible maximal finite subgroups of $GL_n(\mathcal{D})$ with $n \in \{1,3\}$. \\

Suppose first that $n=1.$ Then we have the following fact.
\begin{theorem}[{\cite[Theorem 6.1]{8}}]\label{aimf of GL1}
Let $G $ be an absolutely irreducible maximal finite subgroup of $GL_1(\mathcal{D})=\mathcal{D}^{\times}$ for some totally definite quaternion algebra $\mathcal{D}$ over a totally real field $K$. Then: \\
(a) $K$ is the maximal totally real subfield of a cyclotomic field. \\
(b) If $[K : \mathbb{Q}] \leq 2,$ then $G$ is one of the following groups (up to isomorphism):
\begin{center}
  \begin{tabular}{ccc}
\hline
$$ & $G$ & $\mathcal{D}$ \\
\hline
$\sharp 1$ & $\mathfrak{T}^* \cong SL_2(\F_3)$ & $D_{2,\infty}$ \\
\hline
$\sharp 2$ & $\widetilde{\textrm{Sym}_3} =\textrm{Dic}_{12}$ & $D_{3,\infty}$  \\
\hline
$\sharp 3$ & $\widetilde{\textrm{Sym}_4} = \mathfrak{O}^*$ & $D_{\sqrt{2},\infty}$  \\
\hline
$\sharp 4$ & $Q_{24}=\textrm{Dic}_{24} $ & $D_{\sqrt{3},\infty}$    \\
\hline
$\sharp 5$ & $\mathfrak{I}^* \cong SL_2(\F_5) $ & $D_{\sqrt{5},\infty}$ \\
\hline
\end{tabular}
\vskip 4pt
\textnormal{Table 4}
\end{center}
where $D_{\alpha,\infty}$ (for $\alpha \in \{\sqrt{2},\sqrt{3},\sqrt{5} \})$ denotes the definite quaternion algebra over $K$ such that $K=\mathbb{Q}(\alpha)$ and $D_{\alpha, \infty}$ is ramified only at the infinite places of $K.$
\end{theorem}




Next, we suppose that $n=3$. For the purpose of this paper, we only consider the case when $K=\Q$. Then the following important result is obtained.

\begin{theorem}[{\cite[Theorem 13.1]{8}}]\label{prim aimf of GL3}
Let $G $ be a primitive absolutely irreducible maximal finite subgroup of $ GL_3(\mathcal{D})$ for some totally definite quaternion algebra $\mathcal{D}$ over $\Q$. Then $\mathcal{D}$ is one of $D_{2,\infty}, D_{3,\infty},$ or $D_{7,\infty}$ and $G$ is (conjugate to) one of the following groups:
\begin{center}
 \begin{tabular}{ccc}
\hline
$$ & $G$ & $\mathcal{D}$ \\
\hline
$\sharp 1$ & $SL_2(\F_5)$ & $D_{2,\infty}$ \\
\hline
$\sharp 2$ & $\pm U_3(3)$ & $D_{3,\infty}$  \\
\hline
$\sharp 3$ & $\pm 3_{+}^{1+2}.GL_2(\F_3)$ & $D_{3,\infty}$  \\
\hline
$\sharp 4$ & $ \pm L_2(7).2$ & $D_{7,\infty}$    \\
\hline
\end{tabular}
\vskip 4pt
\textnormal{Table 5}
\end{center}
where $\pm H$ (for a finite matrix group $H$) denotes the group generated by $H$ and the negative identity matrix.
\end{theorem}


Finally, in view of Remark~\ref{impri and prim} and Theorem~\ref{aimf of GL1}, we also have the following result.
\begin{theorem}\label{imprim aimf GL2}
Let $G $ be an imprimitive absolutely irreducible maximal finite subgroup of $ GL_3(\mathcal{D})$ for some totally definite quaternion algebra $\mathcal{D}$ over $\Q$. Then $\mathcal{D}$ is either $D_{2,\infty}$ or $D_{3,\infty}$ and $G$ is one of the following groups:
\begin{center}
 \begin{tabular}{ccc}
\hline
$$ & $G$ & $\mathcal{D}$ \\
\hline
$\sharp 1$ & $ SL_2(\F_3) \wr \textrm{Sym}_3$ & $D_{2,\infty}$ \\
\hline
$\sharp 2$ & $ \textrm{Dic}_{12} \wr \textrm{Sym}_3$ & $D_{3,\infty}$  \\
\hline
\end{tabular}
\vskip 4pt
\textnormal{Table 6}
\end{center}
\end{theorem}

We conclude this section by introducing reducible maximal finite subgroups of $GL_3 (\mathcal{D})$ with $\mathcal{D}=D_{p,\infty}$ for some prime $p.$ To this aim, let $G \leq GL_3(\mathcal{D})$ be a reducible maximal finite subgroup. Then in view of Remark~\ref{irred max} (together with \cite[Remark II.4]{Ples}), we have the following two cases: \\
\indent (Type I) $G$ is conjugate to a diagonal block matrix $\textrm{Diag}(G_1, G_2)$ where $G_1$ (resp.\ $G_2$) is an irreducible maximal finite subgroup of $GL_2(\mathcal{D})$ (resp.\ $\mathcal{D}^{\times}$). \\
\indent (Type II) $G$ is conjugate to a diagonal block matrix $\textrm{Diag}(G_1, G_2, G_3)$ where $G_1, G_2, G_3$ are irreducible maximal finite subgroups of $\mathcal{D}^{\times}$. \\
\indent We examine each case in the following subsequent theorems.
\begin{theorem}\label{type I red new}
Let $G$ be a reducible maximal finite subgroup of $GL_3(\mathcal{D})$ of (Type I) above with $\mathcal{D}=D_{p,\infty}$ for some prime $p$. Then $G = G_1 \times G_2$ (up to isomorphism) with the following specified groups $G_1$ and $G_2$: \\
(1) $G_1 \in \{2_{-}^{1+4}.\textrm{Alt}_5, SL_2(\F_3 ) \times \textrm{Sym}_3 \}$ and $G_2 = \mathfrak{T}^*$;\\ 
(2) $G_1 \in \{SL_2(\F_9), C_3 : (SL_2(\F_3).2), \mathfrak{T}^* \rtimes C_4 \}$ and $G_2 =\textrm{Dic}_{12}$;\\ 
(3) $G_1 \in \{SL_2(\F_5).2, SL_2(\F_5):2, \mathfrak{T}^* \times C_3\}$ and $G_2 =C_6$; \\ 
(4) $G_1 \in \{GL_2(\F_3), \mathfrak{I}^*, \mathfrak{O}^* \}$ and $G_2 =C_2$;\\ 
(5) $G_1 \in \{ GL_2(\F_3), C_{12} \rtimes C_2, \mathfrak{T}^* \rtimes C_4, \textrm{Dic}_{24}, \mathfrak{O}^* \}$ and $G_2 =C_4$;\\ 
(6) $G_1 \in \{GL_2(\F_3), \mathfrak{O}^*, \textrm{Dic}_{24}, \mathfrak{I}^* \}$ and $G_2 =C_6$;\\ 
(7) $G_1 = \mathfrak{T}^* \rtimes C_4$ and $G_2 =C_6$; \\ 
(8) $G_1 \in \{\mathfrak{T}^* \times C_3, \textrm{Dic}_{12} \rtimes C_6\}$ and $G_2 = C_4$; \\ 
(9) $G_1=\mathfrak{I}^*$ and $G_2 = C_4$;\\ 
(10) $G_1 \in \{D_6, \mathfrak{T}^*, \textrm{Dic}_{12}\}$ and $G_2 = C_2$.  
\end{theorem}
\begin{proof}
We provide a detailed proof for (1), and then, we can proceed in a similar fashion with the specified choices for (2)-(10). \\
\indent (1) Take $p=2$ so that $\mathcal{D}=D_{2,\infty}.$ Among all irreducible finite subgroups of $GL_{2}(D_{2,\infty}),$ only the three groups $2_{-}^{1+4}.\textrm{Alt}_5, SL_{2}(\F_3) \times \textrm{Sym}_3,$ and $\mathfrak{T}^* \wr \textrm{Sym}_2$ are maximal (up to isomorphism). (For the list of such irreducible finite subgroups, see \cite[$\S$4]{4}.) Also, recall from \cite[Theorem 6.1]{8} that the group $\mathfrak{T}^*$ is an absolutely irreducible maximal finite subgroup of $D_{2,\infty}^{\times}.$ Then case (1) follows from the observation that $(\mathfrak{T}^* \wr \textrm{Sym}_2) \times \mathfrak{T}^* \leq \mathfrak{T}^* \wr \textrm{Sym}_3,$ the latter being an imprimitive absolutely irreducible maximal finite subgroup of $GL_3(D_{2,\infty})$ (see Theorem~\ref{imprim aimf GL2}). \\
\indent (2) Take $p=3$. Then among all irreducible finite subgroups of $GL_2(D_{3,\infty})$, only the four groups $SL_2(\F_9), C_3:(SL_2(\F_3).2), \textrm{Dic}_{12} \wr \textrm{Sym}_2,$ and $\mathfrak{T}^* \rtimes C_4$ are maximal, and $\textrm{Dic}_{12}$ is an absolutely irreducible maximal finite subgroup of $D_{3,\infty}^{\times}.$ (In fact, we have $(\textrm{Dic}_{12} \wr \textrm{Sym}_2) \times \textrm{Dic}_{12} \leq \textrm{Dic}_{12} \wr \textrm{Sym}_3,$ while the other three groups are maximal finite subgroups of $GL_3(D_{3,\infty})$.) \\
\indent (3) Take $p=5$. Then among all irreducible finite subgroups of $GL_2(D_{5,\infty})$, only the three groups $SL_2(\F_5).2, SL_2(\F_5):2,$ and $\mathfrak{T}^* \times C_3$ are maximal, and $C_6$ is an irreducible maximal finite subgroup of $D_{5,\infty}^{\times}.$  \\
\indent (4) Take $p=37$. Then among all irreducible finite subgroups of $GL_2(D_{37,\infty})$, only the three groups $GL_2(\F_3), \mathfrak{I}^*$, and $\mathfrak{O}^*$ are maximal, and $C_2$ is an irreducible maximal finite subgroup of $D_{37,\infty}^{\times}.$  \\
\indent (5) For the groups $G_1 = GL_2(\F_3), C_{12} \rtimes C_2, \mathfrak{T}^* \rtimes C_4$, or $\textrm{Dic}_{24},$ take $p=31$. Then all those four groups are irreducible finite subgroups of $GL_{2}(D_{31,\infty})$, which are (the only) maximal, and $C_4$ is an irreducible maximal finite subgroup of $D_{31,\infty}^{\times}.$ \\
For the group $G_1 = \mathfrak{O}^*$, take $p=19$, and note that $\mathfrak{O}^*$ is an irreducible finite subgroup of $GL_2(D_{19,\infty})$, that is maximal (along with three other maximal finite subgroups $C_{12} \rtimes C_2, \mathfrak{T}^* \rtimes C_4, \textrm{Dic}_{24}$), and $C_4$ is an irreducible maximal finite subgroup of $D_{19,\infty}^{\times}.$ \\
\indent (6) Take $p=53$. Then among all irreducible finite subgroups of $GL_2(D_{53,\infty})$, the six groups $\mathfrak{T}^* \times C_3, \textrm{Dic}_{12} \rtimes C_6, GL_2(\F_3), \mathfrak{O}^*, \textrm{Dic}_{24},$ and $\mathfrak{I}^*$ are maximal, and $C_6$ is an irreducible maximal finite subgroup of $D_{53,\infty}^{\times}.$ \\
\indent (7)-(8) Take $p=71$. Then among all irreducible finite subgroups of $GL_2(D_{71,\infty})$, the four groups $C_{12} \rtimes C_2, \mathfrak{T}^* \rtimes C_4, \mathfrak{T}^* \times C_3,$ and $\textrm{Dic}_{12} \rtimes C_6$ are maximal, and $C_4$ and $C_6$ are irreducible maximal finite subgroups of $D_{71,\infty}^{\times}.$ (In fact, if we fix $G_2=C_6,$ then only two groups $(\mathfrak{T}^* \rtimes C_4) \times C_6$ and $(\mathfrak{T}^* \times C_3) \times C_6$ are maximal finite subgroups of $GL_3(D_{71,\infty})$. (Note that the latter group is in case (3).) Also, if we fix $G_2 = C_4,$ then the two groups $(\mathfrak{T}^* \times C_3) \times C_4$ and $(\textrm{Dic}_{12} \rtimes C_6) \times C_4$ are maximal finite subgroups of $GL_3(D_{71,\infty})$, while the other two groups are in case (5).)   \\
\indent (9) Take $p=43$. Then among all irreducible finite subgroups of $GL_2(D_{43,\infty})$, the five groups $C_{12} \rtimes C_2, \mathfrak{T}^* \rtimes C_4, \mathfrak{O}^*, \textrm{Dic}_{24}$, and $\mathfrak{I}^*$ are maximal, and $C_4$ is an irreducible maximal finite subgroup of $D_{43,\infty}^{\times}.$ (In fact, the group $\mathfrak{I}^* \times C_4 $ is a maximal finite subgroup of $GL_3(D_{43,\infty})$, while the other four groups are in case (5).)   \\
\indent (10) Take $p=241$. Then all the four irreducible finite subgroups $D_4, D_6, \textrm{Dic}_{12},$ and $\mathfrak{T}^*$ of $GL_2(D_{241,\infty})$ are maximal, and $C_2$ is an irreducible maximal finite subgroup of $D_{241,\infty}^{\times}.$ (In fact, we have $D_4 \times C_2 \leq C_2 \wr \textrm{Sym}_3$, the latter being an irreducible finite subgroup of $GL_3(D_{241,\infty})$, while the other three groups are maximal finite subgroups of $GL_3(D_{241,\infty})$.)   \\
\indent This completes the proof.
\end{proof}

\begin{theorem}\label{type II red}
There exists no reducible maximal finite subgroup of $\textrm{GL}_3(\mathcal{D})$ of (Type II) above with $\mathcal{D}=D_{p,\infty}$ for some prime $p.$
\end{theorem}
\begin{proof}
In view of Theorem~\ref{aimf of GL1} and the observation that the cyclic groups $C_6, C_4,$ and $C_2$ are irreducible maximal finite subgroups of $D_{p,\infty}^{\times}$ for some corresponding prime $p \geq 5$, we can see that the possible maximal reducible finite subgroups $G$ of $GL_3(D_{p,\infty})$ of (Type II) are of the form $G=G_1^3$ for $G_1 \in \{\mathfrak{T}^*, \textrm{Dic}_{12}, C_6, C_4, C_2\}$ or $G=G_1 \times G_2 \times G_3$ for $G_1, G_2, G_3 \in \{C_4, C_6\}$ (depending on $p$). If $G=(\mathfrak{T}^*)^3$ (resp.\ $G=\textrm{Dic}_{12}^3$) so that $p=2$ (resp.\ $p=3$), then $G$ is isomorphic to a subgroup of $SL_2(\F_3) \wr \textrm{Sym}_3$ (resp.\ $\textrm{Dic}_{12} \wr \textrm{Sym}_3$), which is the imprimitive absolutely irreducible maximal finite subgroup of $GL_3(D_{2,\infty})$ (resp.\ $GL_3(D_{3,\infty}))$. If $G=C_6^3$ so that $C_6 \leq D_{p,\infty}^{\times},$ then $G \leq C_6 \wr \textrm{Sym}_3$, the latter being an irreducible finite subgroup of $GL_3(D_{p,\infty}).$ Similar arguments apply for the cases when $G=C_4^3$ or $G=C_2^3.$ Finally, if $G=C_6^2 \times C_4$ (resp.\ $G=C_6 \times C_4^2)$ so that $C_4, C_6 \leq D_{p,\infty}^{\times},$ then $G \leq  (\textrm{Dic}_{12} \rtimes C_6 ) \times C_4$ (resp.\ $G \leq (\mathfrak{T}^* \rtimes C_4) \times C_6)$, which is a reducible finite subgroup of $GL_3(D_{p,\infty})$ of (Type I). \\
\indent This completes the proof.
\end{proof}

\section{Main Result}\label{main}
In this section, we give a classification of finite groups that can be realized as the automorphism group of a polarized abelian threefold over a finite field which is maximal in its isogeny class in the following sense.

\begin{definition}\label{def 20}
  Let $X$ be an abelian variety over a field $k$, and let $G$ be a finite group. Suppose that the following two conditions hold: \\
\indent (i) (realizability) there exists an abelian variety $X^{\prime}$ over $k$ that is $k$-isogenous to $X$ with a polarization $\mathcal{L}$ such that $G=\textrm{Aut}_k(X^{\prime},\mathcal{L}),$ and \\
\indent (ii) (maximality) there is no finite group $H$ such that $G$ is isomorphic to a proper subgroup of $H$ and $H=\textrm{Aut}_k (Y,\mathcal{M})$ for some abelian variety $Y$ over $k$ that is $k$-isogenous to $X$ with a polarization $\mathcal{M}.$ \\
\indent In this case, $G$ is said to be \emph{realizable maximally (or maximal, in short) in the isogeny class of $X$} as the full automorphism group of a polarized abelian variety over $k$.
\end{definition}

The resulting classification is much more complicated than that of polarized abelian surfaces case \cite[$\S6$]{4}. In the sequel, $G$ will always denote a finite group. Following \cite{5}, we are ready to introduce one of the main results of this section.

\begin{theorem}[{\cite[Theorem 4.1]{5}}]\label{thm old 24}
  There exists a finite field $k$ and a simple abelian threefold $X$ over $k$ such that $G$ is the automorphism group of a polarized abelian threefold over $k,$ which is maximal in the isogeny class of $X$ if and only if $G$ is one of the cyclic groups $C_n$ for $n \in \{2,4,6,14,18\}$ (up to isomorphism).
\end{theorem}


For the rest of this section, we take care of various cases of non-simple abelian threefolds in view of the decomposition of abelian varieties as the product of powers of simple abelian varieties, up to $k$-isogeny (see $\S2.1$ above). The next theorem takes care of the cases in which our abelian threefold is isogenous to the product of a simple abelian surface and an elliptic curve.

\begin{theorem}\label{prodnonisoellip}
There exist a finite field $k$, a simple abelian surface $Y$, and an elliptic curve $E$ over $k$ such that $G$ is the automorphism group of a polarized abelian threefold over $k,$ which is maximal in the isogeny class of $X:=Y \times E$ if and only if $G=G_1 \times G_2$ is one of the following groups (up to isomorphism) with the specified groups $G_1$ and $G_2$: \\ 
(1) $G_1 = \mathfrak{I}^* $ and $G_2 \in \{C_2, C_4 \}$; \\
(2) $G_1 = \mathfrak{O}^* $ and $G_2 \in \{C_2, C_4 \}$; \\
(3) $G_1 = \textrm{Dic}_{24}$ and $G_2 \in \{C_2, C_6 \}$; \\
(4) $G_1 = \mathfrak{T}^* $ and $G_2 \in \{C_2, C_4, C_6 \}$; \\
(5) $G_1 = \textrm{Dic}_{12} $ and $G_2 \in \{C_2, C_4 \}$; \\
(6) $G_1 \in \{C_8,C_{10},C_{12}\}$ and $G_2 \in \{C_2, C_4, C_6, \textrm{Dic}_{12}, \mathfrak{T}^* \}$; \\
(7) $G_1 = C_6 $ and $G_2 \in \{C_2, C_4, C_6 \}$; \\
(8) $G_1 = C_4 $ and $G_2 \in \{C_2, C_4 \}$; \\
(9) $G_1 = C_2$ and $G_2 = C_2.$
\end{theorem}
\begin{proof}
Suppose first that there exist a finite field $k$, a simple abelian surface $Y$, and an elliptic curve $E$ over $k$ such that $G$ is the automorphism group of a polarized abelian threefold over $k$, which is maximal in the isogeny class of $X:=Y \times E.$ In particular, we have $\textrm{End}_k^0(X)=\textrm{End}_k^0(Y) \oplus \textrm{End}_k^0(E).$ Since $G$ is a maximal finite subgroup of $\textrm{End}_k^0(X)$ by assumption, it follows from Goursat's Lemma, \cite[Theorem 6.5 and Corollary 3.4]{4}, and Remark~\ref{exclude rmk1} below that $G$ must be one of the 32 groups (up to isomorphism) in the statement of the theorem. Hence, it suffices to show the converse. We prove the converse by considering them one by one. First, we provide a detailed proof for the groups in case (1).\\
\indent (1) Take $G=\mathfrak{I}^* \times C_2$. Let $\pi$ be a zero of the quadratic polynomial $h:=t^2 -5 \in \Z[t]$ so that there is a simple abelian surface $Y$ over $k=\F_5$ such that $\textrm{End}_k^0(Y)=D_{2,\infty} \otimes_{\Q} \Q(\sqrt{5})$. (For the existence of such a $Y$, see the proof of \cite[Theorem 6.5]{4}.) Also, note that there is an ordinary elliptic curve $E$ over $k$ such that $\textrm{End}_k^0(E)=\Q(\sqrt{-19})$ by \cite[Theorem 4.1]{12}. Let $X=Y \times E.$ Then we have $D:=\textrm{End}_k^0(X)=(D_{2,\infty} \otimes_{\Q} \Q(\sqrt{5})) \oplus \Q(\sqrt{-19}).$ Let $\mathcal{O}_1$ be a maximal $\Z$-order in $D_{2,\infty} \otimes_{\Q} \Q(\sqrt{5})$ with $\mathfrak{I}^* \leq \mathcal{O}_1^{\times}$, and let $\mathcal{O}_2 = \Z \left[\frac{1+\sqrt{-19}}{2}\right].$ Then $\mathcal{O}:=\mathcal{O}_1 \oplus \mathcal{O}_2$ is a maximal $\Z$-order in $D$ by Theorem~\ref{max gen}, and hence, there exists an abelian threefold $X^{\prime}$ over $k$ such that $X^{\prime}$ is $k$-isogenous to $X$ and $\textrm{End}_k(X^{\prime})=\mathcal{O}$ by \cite[Theorem 3.13]{12}. By Goursat's lemma, $G$ is a maximal finite subgroup of $\mathcal{O}^{\times}$.\\
\indent   Now, let $\mathcal{L}$ be an ample line bundle on $X^{\prime}$, and put $\displaystyle \mathcal{L}^{\prime}:=\bigotimes_{f \in G} f^* \mathcal{L}.$ Then $\mathcal{L}^{\prime}$ is also an ample line bundle on $X^{\prime}$ that is preserved under the action of $G$ so that $G \leq \textrm{Aut}_k(X^{\prime},\mathcal{L}^{\prime}).$ Since $\textrm{Aut}_k (X^{\prime},\mathcal{L}^{\prime})$ is a finite subgroup of $\textrm{Aut}_k(X^{\prime})=\mathcal{O}^{\times},$ it follows from the maximality of $G$ that $G=\textrm{Aut}_k(X^{\prime},\mathcal{L}^{\prime}).$ \\
\indent  Now, suppose that $Z$ is an abelian threefold over $k$ which is $k$-isogenous to $X.$ In particular, we have $\textrm{End}_k^0(Z)=(D_{2,\infty} \otimes_{\Q} \Q(\sqrt{5})) \oplus \Q(\sqrt{-19}).$ Suppose that there is a finite group $H$ such that $H=\textrm{Aut}_k(Z, \mathcal{M})$ for a polarization $\mathcal{M}$ on $Z,$ and $G$ is isomorphic to a proper subgroup of $H.$ Then $H$ is a finite subgroup of $D^{\times},$ and hence, it follows from the maximality of $G$ as a finite subgroup of $D^{\times}$ that $H=G$, which is a contradiction. Thus, we can conclude that $G$ is maximal in the isogeny class of $X.$ \\ \indent For the group $G=\mathfrak{I}^* \times C_4,$ we can proceed as above with the choices of $h=t^2 -5 \in \Z[t]$ and an ordinary elliptic curve $E$ over $k=\F_5$ with $\textrm{End}_k^0(E)=\Q(\sqrt{-1}).$ \\

Then, for the groups $G$ in (2)-(9), we proceed in a similar argument as in (1) with the following choices for a simple abelian surface $Y$ over a finite field $k$, and an elliptic curve $E$ over $k$ satisfying the endomorphism conditions specified below.

\begin{longtable}{|c||c|c|c|c|c|c|}
  \hline
   & $G$ & $h\in \Z[t]$ & $k$ & $\textrm{End}_k^0(Y)$ &  $\textrm{End}_k^0(E)$ \\\hline\hline
  \multirow{2}{*}{(2)} & $\mathfrak{O}^* \times C_2$ & \multirow{2}{*}{$t^2-2$} & \multirow{2}{*}{$\F_2$} & \multirow{2}{7em}{$D_{2,\infty} \otimes_{\Q} \Q(\sqrt{2})$} & $\Q(\sqrt{-2})$ \\
      & $\mathfrak{O}^* \times C_4$ &  &  &    & $\Q(\sqrt{-1})$ \\\hline
   \multirow{2}{*}{(3)} & $\textrm{Dic}_{24} \times C_2$ & \multirow{2}{*}{$t^2-3$}&  \multirow{2}{*}{$\F_3$} &  \multirow{2}{*}{$D_{3,\infty} \otimes_{\Q} \Q(\sqrt{3})$} & $\Q(\sqrt{-2})$  \\
      &$\textrm{Dic}_{24} \times C_6$ &  & &   & $\Q(\sqrt{-3})$  \\\hline
 \multirow{3}{*}{(4)} & $\mathfrak{T}^* \times C_2$ &  \multirow{2}{*}{$t^2-3$}&  \multirow{2}{*}{$\F_3$} &  \multirow{2}{*}{$D_{3,\infty} \otimes_{\Q} \Q(\sqrt{3})$} &  $\Q(\sqrt{-2})$  \\
 & $\mathfrak{T}^* \times C_6$ &  & &    & $\Q(\sqrt{-3})$ \\\cline{2-7}
& $\mathfrak{T}^* \times C_4$ & $t^2 -13$ & $\F_{13}$ & $D_{13,\infty} \otimes_{\Q} \Q(\sqrt{13})$   & $\Q(\sqrt{-1})$ \\\hline
 \multirow{2}{*}{(5)} & $\textrm{Dic}_{12} \times C_2$ &  $t^2-11$& $\F_{11}$ &  $D_{11,\infty} \otimes_{\Q} \Q(\sqrt{11})$ &  $\Q(\sqrt{-2})$  \\
 & $\textrm{Dic}_{12} \times C_4$ &$t^2-17$ &$\F_{17}$ &  $D_{17,\infty} \otimes_{\Q} \Q(\sqrt{17}) $& $\Q(\sqrt{-1})$\\\hline
 \multirow{15}{*}{(6)} & $C_8 \times C_2$ & \multirow{4}{*}{$t^4+16$} & \multirow{4}{*}{$\F_4$} & \multirow{4}{*}{$\Q(\zeta_8)$} &  $\Q(\sqrt{-15})$ \\
      & $C_8 \times C_4$ &  &  &    & $\Q(\sqrt{-1})$ \\
     & $C_8 \times C_6$ &  &  &    & $\Q(\sqrt{-3})$ \\
       & $C_8 \times \mathfrak{T}^*$  &  &  &    & $D_{2,\infty}$ \\\cline{2-7}
      & $C_8 \times \textrm{Dic}_{12}$ & $t^4+81$ & $\F_{9}$  & $\Q(\zeta_8)$ & $D_{3,\infty}$ \\\cline{2-7}
  & $C_{10} \times C_2$ & \multirow{2}{*}{$t^4-5t^3+25t^2-125t+625$} & \multirow{2}{*}{$\F_{25}$} & \multirow{2}{*}{$\Q(\zeta_{10})$} & $\Q(\sqrt{-11})$ \\
      & $C_{10} \times C_6$ & &  &   & $\Q(\sqrt{-3})$  \\\cline{2-7}
      & $C_{10} \times C_4$ &\multirow{2}{*}{$t^4-3t^3+9t^2-27t+81$} & \multirow{2}{1em}{$\F_{9}$} &\multirow{2}{*}{$\Q(\zeta_{10})$}  & $\Q(\sqrt{-1})$\\
      & $C_{10} \times \textrm{Dic}_{12}$ &  &  &   & $D_{3,\infty}$ \\\cline{2-7}
      & $C_{10} \times \mathfrak{T}^*$ & $t^4-2t^3+4t^2-8t+16$ & $\F_{4}$  & $\Q(\zeta_{10})$ &   $D_{2,\infty}$ \\\cline{2-7}
  & $C_{12}\times C_2$ & \multirow{4}{*}{$t^4-9t^2+81$} & \multirow{4}{*}{$\F_9$} & \multirow{4}{*}{$\Q(\zeta_{12})$}  & $\Q(\sqrt{-2})$ \\
      &$C_{12}\times C_4$ &  &  &   & $\Q(\sqrt{-1})$ \\
      & $C_{12}\times C_6$ &  &  &    & $\Q(\sqrt{-3})$ \\
      & $C_{12} \times \textrm{Dic}_{12}$ &  &   &    & $D_{3,\infty}$ \\\cline{2-7}
      & $C_{12} \times \mathfrak{T}^*$ & $t^4-4t^2+16$ & $\F_{4}$  & $\Q(\zeta_{12})$ &   $D_{2,\infty}$ \\\hline
\multirow{3}{*}{(7)} & $C_6 \times C_2$ & \multirow{2}{*}{$t^4+2t^2+25$} & \multirow{2}{*}{$\F_5$} & \multirow{2}{7em}{$\Q(\sqrt{2}+\sqrt{-3})$} &   $\Q(\sqrt{-19})$ \\
 & $C_6 \times C_4$ &  &  &    & $\Q(\sqrt{-1})$ \\\cline{2-7}
      & $C_6 \times C_6$ & $t^4+10t^2+361$ & $\F_{19}$  & $\Q(\sqrt{7}+2\sqrt{-3})$ &  $\Q(\sqrt{-3})$ \\\hline
\multirow{2}{*}{(8)} & $C_4 \times C_2$ & $t^4-10t^2+49$ & $\F_7$ & $\Q(\sqrt{6}+\sqrt{-1})$ &  $\Q(\sqrt{-6})$ \\
      & $C_4 \times C_4$ & $t^4-18t^2+289$ & $\F_{17}$  & $\Q(\sqrt{13}+2\sqrt{-1})$ &   $\Q(\sqrt{-1})$ \\\hline
 (9) & $C_2\times C_2$  & $t^4-6t^2+49$ & $\F_7$ & $\Q(\sqrt{5}+\sqrt{-2})$ &  $\Q(\sqrt{-6})$ \\ \hline
\end{longtable}
We remark that $D_{3,\infty} \otimes_{\Q} \Q(\sqrt{3}) \cong D_{2,\infty} \otimes_{\Q} \Q(\sqrt{3})$, $D_{13,\infty} \otimes_{\Q} \Q(\sqrt{13}) \cong D_{2,\infty} \otimes_{\Q} \Q(\sqrt{13})$, $D_{11,\infty} \otimes_{\Q} \Q(\sqrt{11}) \cong D_{3,\infty} \otimes_{\Q} \Q(\sqrt{11})$, and   $D_{17,\infty} \otimes_{\Q} \Q(\sqrt{17}) \cong D_{3,\infty} \otimes_{\Q} \Q(\sqrt{17})$. \\
\indent This completes the proof.
\end{proof}

\begin{remark}\label{exclude rmk1}
Among all possible $45$ combinations of $G:=G_1 \times G_2 $ with
\begin{equation*}
G_1 \in \{C_2, C_4, C_6, C_8, C_{10}, C_{12}, \textrm{Dic}_{12}, \mathfrak{T}^* , \textrm{Dic}_{24}, \mathfrak{O}^*, \mathfrak{I}^* \}
\end{equation*}
and $G_2 \in \{C_2, C_4, C_6, \textrm{Dic}_{12}, \mathfrak{T}^*\}$, up to isomorphism, there are 13 groups which cannot be indeed realized as the automorphism group of a polarized abelian threefold over a finite field $k,$ which is maximal in the isogeny class of the product a simple abelian surface and an elliptic curve over $k.$ For example, the groups $G:= G_1 \times \textrm{Dic}_{12}$ with $G_1 \in \{\mathfrak{I}^*, \mathfrak{O}^* \}$ cannot occur due to the issue of the characteristic of the base field $k.$ On the other hand, the group $G:=\textrm{Dic}_{24} \times \textrm{Dic}_{12}$ cannot be realized due to the fact that the latter group can occur only when the cardinality of the base field $k$ is an even power of $p=3$, while the former group occurs only when the cardinality of $k$ is an odd power of its characteristic.
\end{remark}

If our abelian threefold happens to be $k$-isogenous to the product of three non-isogenous elliptic curves over a finite field $k,$ then we have the following result.
\begin{theorem}\label{prodnonisoellip2}
There exist a finite field $k$ and three non-isogenous elliptic curves $E_1,E_2,E_3$ over $k$ such that $G$ is the automorphism group of a polarized abelian threefold over $k,$ which is maximal in the isogeny class of $X:=E_1 \times E_2 \times E_3$ if and only if $G=G_1 \times G_2 \times G_3$ is one of the following groups (up to isomorphism) with the specified groups $G_1$, $G_2$, and $G_3$: \\
(1) $G_1 =G_2 = C_2$ and $G_3 \in \{C_2, C_4, C_6, \textrm{Dic}_{12}, \mathfrak{T}^* \}$; \\
(2) $G_1 =C_2, G_2=C_4$, and $G_3 \in \{C_4, C_6 , \textrm{Dic}_{12}, \mathfrak{T}^* \}$; \\
(3) $G_1 =C_2, G_2 = C_6,$ and $G_3 \in \{C_6, \textrm{Dic}_{12}, \mathfrak{T}^* \}$; \\
(4) $G_1 \in \{C_2, C_4, C_6\}$ and $G_2=G_3= \textrm{Dic}_{12}$; \\
(5) $G_1 \in \{C_2, C_4, C_6\}$ and $G_2 =G_3 =\mathfrak{T}^*$; \\
(6) $G_1 =G_2 = C_4$ and $G_3 \in \{C_4, C_6\}$; \\
(7) $G_1 = C_4, G_2 = C_6,$ and $G_3 \in \{C_6, \textrm{Dic}_{12}, \mathfrak{T}^* \}$; \\
(8) $G_1 =G_2 = C_6$ and $G_2 \in \{C_6, \textrm{Dic}_{12}, \mathfrak{T}^* \}$.
\end{theorem}


\begin{proof}
Suppose first that there exist a finite field $k$ and three non-isogenous elliptic curves $E_1,E_2,E_3$ over $k$ such that $G$ is the automorphism group of a polarized abelian threefold over $k,$ which is maximal in the isogeny class of $X:=E_1 \times E_2 \times E_3$. In particular, we have $\textrm{End}_k^0(X)=\textrm{End}_k^0(E_1) \oplus \textrm{End}_k^0(E_2) \oplus \textrm{End}_k^0(E_3).$ Since $G$ is a maximal finite subgroup of $\textrm{End}_k^0(X)$ by assumption, it follows from Goursat's Lemma, \cite[Corollary 3.4]{4}, and Remark~\ref{exclude rmk2} that $G$ must be one of the 26 groups (up to isomorphism) in the statement of the theorem. Hence, it suffices to show the converse. We prove the converse by considering them one by one. First, we provide a detailed proof for the groups in case (1).\\
\indent (1) Take $G=C_2 \times C_2 \times C_2$. Let $k=\F_3.$ Then there are three non-isogenous elliptic curves $E_1, E_2,$ and $E_3$ over $k$ such that $\textrm{End}_k^0(E_1)=\Q(\sqrt{-2})$ and $\textrm{End}_k^0(E_2)=\textrm{End}_k^0(E_3)=\Q(\sqrt{-11})$ by \cite[Theorem 4.1]{12}. Let $X=E_1 \times E_2 \times E_3$ so that $\textrm{End}_k^0(X)=D:=\Q(\sqrt{-2})\oplus \Q(\sqrt{-11}) \oplus \Q(\sqrt{-11}).$ By Theorem~\ref{max gen}, $\mathcal{O}:=\Z[\sqrt{-2}]\oplus \Z \left[\frac{1+\sqrt{-11}}{2}\right] \oplus \Z \left[\frac{1+\sqrt{-11}}{2}\right]$ is a maximal $\Z$-order in $D$, and hence, there exists an abelian threefold $X^{\prime}$ over $k$ such that $X^{\prime}$ is $k$-isogenous to $X$ and $\textrm{End}_k(X^{\prime})=\mathcal{O}$ by \cite[Theorem 3.13]{12}. Then it follows that $G \cong \mathcal{O}^{\times}=\textrm{Aut}_k(X^{\prime})$ (with the principal product polarization). It is also easy to see that $G$ is maximal in the isogeny class of $X.$ \\
\indent For the group $G=C_2 \times C_2 \times C_4,$ we can proceed as above with the choices $k=\F_5$ and three non-isogenous elliptic curves $E_1,E_2,E_3$ with $\textrm{End}_k^0(E_1)=\Q(\sqrt{-5}), \textrm{End}_k^0(E_2)=\Q(\sqrt{-19}),$ and $\textrm{End}_k^0(E_3)=\Q(\sqrt{-1}).$ Similarly:\\
\indent For the group $G=C_2 \times C_2 \times C_6,$ we take $k=\F_7$ and three non-isogenous elliptic curves $E_1,E_2,E_3$ with $\textrm{End}_k^0(E_1)=\textrm{End}_k^0(E_2)=\Q(\sqrt{-6})$ and $\textrm{End}_k^0(E_3)=\Q(\sqrt{-3}).$ \\
\indent For the group $G=C_2 \times C_2 \times \textrm{Dic}_{12},$ we take $k=\F_9$ and three non-isogenous elliptic curves $E_1,E_2,E_3$ with $\textrm{End}_k^0(E_1)=\Q(\sqrt{-2}), \textrm{End}_k^0(E_2)=\Q(\sqrt{-5}),$ and $\textrm{End}_k^0(E_3)=D_{3,\infty}.$ \\
\indent Finally, for the group $G=C_2 \times C_2 \times \mathfrak{T}^*,$ we take $k=\F_4$ and three non-isogenous elliptic curves $E_1,E_2,E_3$ with $\textrm{End}_k^0(E_1)=\Q(\sqrt{-15}), \textrm{End}_k^0(E_2)=\Q(\sqrt{-7}),$ and $\textrm{End}_k^0(E_3)=D_{2,\infty}.$\\

Then, for (2)-(8), we can proceed in a similar argument as in (1) with the following specified choices for the groups $G$, the finite field $k$, and three non-isogenous elliptic curves $E_1,E_2, E_3$ over $k$ with the endomorphism conditions given below.

\begin{longtable}{|c||c|c|c|}
  \hline
  & $G$ & $k$ & $\textrm{End}_k^0(E_1), \textrm{End}_k^0(E_2), \textrm{End}_k^0(E_3)$\\\hline\hline
  \multirow{4}{*}{(2)} & $C_2 \times C_4 \times C_4$ & $\F_5$ & $\Q(\sqrt{-5}), \Q(\sqrt{-1}), \Q(\sqrt{-1})$\\
  & $C_2 \times C_4 \times C_6$ & $\F_4$ & $\Q(\sqrt{-15}), \Q(\sqrt{-1}), \Q(\sqrt{-3})$\\
   & $C_2 \times C_4 \times \textrm{Dic}_{12}$ & $\F_9$ & $\Q(\sqrt{-2}), \Q(\sqrt{-1}),  D_{3,\infty}$\\
    & $C_2 \times C_4 \times \mathfrak{T}^*$ & $\F_4$ & $\Q(\sqrt{-15}), \Q(\sqrt{-1}), D_{2,\infty}$\\\hline
   \multirow{3}{*}{(3)} & $C_2 \times C_6 \times C_6$ & $\F_4$ & $\Q(\sqrt{-15}), \Q(\sqrt{-3}), \Q(\sqrt{-3})$\\
  & $C_2 \times C_6 \times \textrm{Dic}_{12}$ & $\F_9$ & $\Q(\sqrt{-2}), \Q(\sqrt{-3}), D_{3,\infty}$\\
& $C_2 \times C_6 \times \mathfrak{T}^*$ & $\F_4$ & $\Q(\sqrt{-15}), \Q(\sqrt{-3}), D_{2,\infty}$\\\hline
 \multirow{3}{*}{(4)} & $C_2 \times \textrm{Dic}_{12}\times \textrm{Dic}_{12}$ & $\F_9$ & $\Q(\sqrt{-2}), D_{3,\infty}, D_{3,\infty}$\\
 & $C_4 \times \textrm{Dic}_{12}\times \textrm{Dic}_{12}$ & $\F_9$ & $\Q(\sqrt{-1}), D_{3,\infty}, D_{3,\infty}$\\
& $C_6 \times \textrm{Dic}_{12}\times \textrm{Dic}_{12}$ & $\F_9$ & $\Q(\sqrt{-3}), D_{3,\infty}, D_{3,\infty}$\\\hline
\multirow{3}{*}{(5)} & $C_2 \times \mathfrak{T}^*\times \mathfrak{T}^*$ & $\F_4$ & $\Q(\sqrt{-15}), D_{2,\infty}, D_{2,\infty}$\\
& $C_4 \times \mathfrak{T}^*\times \mathfrak{T}^*$ & $\F_4$ & $\Q(\sqrt{-1}), D_{2,\infty}, D_{2,\infty}$\\
& $C_6 \times \mathfrak{T}^*\times \mathfrak{T}^*$ & $\F_4$ & $\Q(\sqrt{-3}), D_{2,\infty}, D_{2,\infty}$\\\hline
\multirow{2}{*}{(6)} & $C_4 \times C_4 \times C_4 $ & $\F_5$ & $\Q(\sqrt{-1}), \Q(\sqrt{-1}), \Q(\sqrt{-1})$\\
 & $C_4 \times C_4 \times C_6 $ & $\F_{13}$ & $\Q(\sqrt{-1}), \Q(\sqrt{-1}), \Q(\sqrt{-3})$\\\hline
 \multirow{3}{*}{(7)} & $C_4 \times C_6 \times C_6 $ & $\F_4$ & $\Q(\sqrt{-1}), \Q(\sqrt{-3}), \Q(\sqrt{-3})$\\
 &  $C_4 \times C_6 \times \textrm{Dic}_{12} $ & $\F_9$ & $\Q(\sqrt{-1}), \Q(\sqrt{-3}), D_{3,\infty}$\\
 & $C_4 \times C_6 \times \mathfrak{T}^* $ & $\F_4$ & $\Q(\sqrt{-1}), \Q(\sqrt{-3}), D_{2,\infty}$\\\hline
 \multirow{3}{*}{(8)} & $C_6 \times C_6 \times C_6 $ & $\F_7$ & $\Q(\sqrt{-3}), \Q(\sqrt{-3}), \Q(\sqrt{-3})$\\
 &  $C_6 \times C_6 \times \textrm{Dic}_{12} $ & $\F_9$ & $\Q(\sqrt{-3}), \Q(\sqrt{-3}), D_{3,\infty}$\\
 & $C_6 \times C_6 \times \mathfrak{T}^* $ & $\F_4$ & $\Q(\sqrt{-3}), \Q(\sqrt{-3}), D_{2,\infty}$\\\hline
  \end{longtable}
This completes the proof.
\end{proof}

\begin{remark}\label{exclude rmk2}
Among all possible $35$ combinations of $G:=G_1 \times G_2 \times G_3$ with $G_1,G_2,G_3 \in \{C_2, C_4, C_6, \textrm{Dic}_{12}, \mathfrak{T}^* \}$, up to isomorphism, there are 9 groups which cannot be indeed realized as the automorphism group of a polarized abelian threefold over a finite field $k,$ which is maximal in the isogeny class of the product of three non-isogenous elliptic curves over $k.$ For example, the groups $G:= G_1 \times \textrm{Dic}_{12} \times \mathfrak{T}^*$ with $G_1 \in \{C_2, C_4, C_6, \textrm{Dic}_{12}, \mathfrak{T}^*\}$ cannot occur due to the issue of the characteristic of the base field $k.$ On the other hand, the groups $G:=\textrm{Dic}_{12} \times \textrm{Dic}_{12} \times \textrm{Dic}_{12}$ and $G:=\mathfrak{T}^* \times \mathfrak{T}^* \times \mathfrak{T}^*$ cannot be realized due to the fact that any two of those three elliptic curves $E_1,E_2,E_3$ are necessarily isogenous to each other.
\end{remark}

The next theorem deals with the case when exactly two of those three elliptic curves in the above are isogenous to each other.
\begin{theorem}\label{prodnonisoellip3}
There exist a finite field $k$ and two non-isogenous elliptic curves $E_1,E_2$ over $k$ such that $G$ is the automorphism group of a polarized abelian threefold over $k,$ which is maximal in the isogeny class of $X:=E_1^2 \times E_2$ if and only if $G=G_1 \times G_2$ is one of the following groups (up to isomorphism) with the specified groups $G_1$ and $G_2$: \\ 
(1) $G_1 \in \{D_4, D_6 \} $ and $G_2 \in \{C_2, C_4, C_6, \textrm{Dic}_{12}, \mathfrak{T}^* \}$;\\ 
(2) $G_1 = \textrm{Dic}_{12}$ and $G_2 \in \{C_2, C_4, C_6, \mathfrak{T}^* \}$; \\ 
(3) $G_1 = SL_2(\mathbb{F}_3)$ and $G_2 \in \{C_2, C_4, C_6 \}$; \\ 
(4) $G_1 \in \{ C_{12} \rtimes C_2, \mathfrak{T}^* \rtimes C_4 \}$ and $G_2 \in \{C_2, C_4, C_6, \textrm{Dic}_{12}, \mathfrak{T}^* \}$; \\ 
(5) $G_1 = GL_{2}(\F_3) $ and $G_2 \in \{C_2, C_4, C_6, \textrm{Dic}_{12} \}$; \\ 
(6) $G_1 \in \{(C_6 \times C_6) \rtimes C_2, \mathfrak{T}^* \times C_3 \}$ and $G_2 \in \{C_2, C_4, C_6, \textrm{Dic}_{12}, \mathfrak{T}^* \}$; \\ 
(7) $G_1 \in \{\mathfrak{I}^*, \textrm{Dic}_{24}, \mathfrak{O}^*\}$ and $G_2 \in \{C_2, C_4, C_6 \}$; \\ 
(8) $G_1 \in \{2^{1+4}_{-}.\textrm{Alt}_5, \textrm{SL}_{2}(\F_3) \times \textrm{Sym}_3, (\textrm{SL}_2(\F_3))^2 \rtimes \textrm{Sym}_2 \}$ and $G_2 \in \{C_2, C_4, C_6, \mathfrak{T}^* \}$; \\ 
(9) $G_1 \in \{\textrm{SL}_2(\F_9), C_3 : (\textrm{SL}_2(\F_3).2), (\textrm{Dic}_{12})^2 \rtimes \textrm{Sym}_2 \} $ and $G_2 \in \{C_2, C_4, C_6, \textrm{Dic}_{12} \}$; \\ 
(10) $G_1 \in \{\textrm{SL}_2(\F_5).2, \textrm{SL}_2(\F_5):2 \} $ and $G_2 \in \{C_2, C_4, C_6 \}$. 
\end{theorem}
\begin{proof}
Suppose first that there exist a finite field $k$ and two non-isogenous elliptic curves $E_1, E_2$ over $k$ such that $G$ is the automorphism group of a polarized abelian threefold over $k$, which is maximal in the isogeny class of $X:=E_1^2 \times E_2.$ In particular, we have $\textrm{End}_k^0(X)=M_2(\textrm{End}_k^0(E_1)) \oplus \textrm{End}_k^0(E_2).$ Since $G$ is a maximal finite subgroup of $\textrm{End}_k^0(X)$ by assumption, it follows from Goursat's Lemma, \cite[Theorems 6.8 and 6.9]{4}, and Remark~\ref{exclude rmk3} below that $G$ must be one of the 80 groups (up to isomorphism) in the statement of the theorem. Hence, it suffices to show the converse. We prove the converse by considering them one by one. First, we provide a detailed proof for the groups in case (1).\\
\indent (1) Take $G=D_4 \times C_2$ (resp.\ $G=D_6 \times C_2$). Let $k=\F_4$. Then there are two non-isogenous elliptic curves $E_1$ and $E_2$ over $k$ such that $\textrm{End}_k^0(E_1)=\textrm{End}_k^0(E_2)=\Q(\sqrt{-15})$ by \cite[Theorem 4.1]{12}. Let $X=E_1^2 \times E_2 $ so that $\textrm{End}_k^0(X)=D:=M_2(\Q(\sqrt{-15}))\oplus \Q(\sqrt{-15}).$ By Theorems~\ref{mat max} and~\ref{max gen}, $\mathcal{O}:= M_2 \left(\Z \left[\frac{1+\sqrt{-15}}{2}\right] \right) \oplus \Z \left[\frac{1+\sqrt{-15}}{2}\right]$ is a maximal $\Z$-order in $D$, and hence, there exists an abelian threefold $X^{\prime}$ over $k$ such that $X^{\prime}$ is $k$-isogenous to $X$ and $\textrm{End}_k(X^{\prime})=\mathcal{O}$ by \cite[Theorem 3.13]{12}. Also, it is known \cite[Table 2]{CN(2014)} that $D_4$ (resp.\ $D_6$) is a maximal finite subgroup of $GL_2 \left(\Z \left[\frac{1+\sqrt{-15}}{2}\right] \right)$, and hence, $G$ is a maximal finite subgroup of $\mathcal{O}^{\times}$ by Goursat's Lemma. \\
\indent Now, let $\mathcal{L}$ be an ample line bundle on $X^{\prime}$, and put $\displaystyle \mathcal{L}^{\prime}:=\bigotimes_{f \in G} f^* \mathcal{L}.$ Then $\mathcal{L}^{\prime}$ is also an ample line bundle on $X^{\prime}$ that is preserved under the action of $G$ so that $G \leq \textrm{Aut}_k(X^{\prime},\mathcal{L}^{\prime}).$ Since $\textrm{Aut}_k (X^{\prime},\mathcal{L}^{\prime})$ is a finite subgroup of $\textrm{Aut}_k(X^{\prime})=\mathcal{O}^{\times},$ it follows from the maximality of $G$ that $G=\textrm{Aut}_k(X^{\prime},\mathcal{L}^{\prime}).$ Furthermore, by a similar argument, we can see that $G$ is maximal in the isogeny class of $X.$  \\
\indent For the groups $G=D_4 \times C_4$ or $G=D_6 \times C_4,$ we can proceed as above with the choices $k=\F_4$ and two non-isogenous elliptic curves $E_1,E_2$ with $\textrm{End}_k^0(E_1)=\Q(\sqrt{-15})$ and $\textrm{End}_k^0(E_2)=\Q(\sqrt{-1}).$ Similarly: \\
 \indent For the groups $G=D_4 \times C_6$ or $G=D_6 \times C_6,$ we take $k=\F_7$ and two non-isogenous elliptic curves $E_1,E_2$ with $\textrm{End}_k^0(E_1)=\Q(\sqrt{-6})$ and $\textrm{End}_k^0(E_2)=\Q(\sqrt{-3}).$ \\
\indent For the groups $G=D_4 \times \textrm{Dic}_{12}$ or $G=D_6 \times \textrm{Dic}_{12},$ we take $k=\F_9$ and two non-isogenous elliptic curves $E_1,E_2$ with $\textrm{End}_k^0(E_1)=\Q(\sqrt{-5})$ and $\textrm{End}_k^0(E_2)=D_{3,\infty}.$ \\
\indent Finally, for the groups $G=D_4 \times \mathfrak{T}^*$ or $G=D_6 \times \mathfrak{T}^*,$ we take $k=\F_4$ and two non-isogenous elliptic curves $E_1,E_2$ with $\textrm{End}_k^0(E_1)=\Q(\sqrt{-15})$ and $\textrm{End}_k^0(E_2)=D_{2,\infty}.$ \\

Then, for (2)-(10), we can proceed in a similar argument as in (1) with the following specified choices for the groups $G$, the finite field $k$, and two non-isogenous elliptic curves $E_1,E_2$ over $k$ with the endomorphism conditions given below.

\begin{longtable}{|c||c|c|c|}
  \hline
  & $G$ & $k$ & $\textrm{End}_k^0(E_1), \textrm{End}_k^0(E_2)$\\\hline\hline
  \multirow{4}{*}{(2)} & $\textrm{Dic}_{12} \times C_2 $ & $\F_{25}$ & $\Q(\sqrt{-21}), \Q(\sqrt{-6})$\\\cline{2-4}
  & $\textrm{Dic}_{12} \times C_4$ & $\F_{25}$ & $\Q(\sqrt{-21}), \Q(\sqrt{-1})$\\\cline{2-4}
  & $\textrm{Dic}_{12} \times C_6$ & $\F_{25}$ & $\Q(\sqrt{-21}), \Q(\sqrt{-3})$\\\cline{2-4}
  & $\textrm{Dic}_{12} \times \mathfrak{T}^*$ & $\F_{4}$ & $\Q(\sqrt{-7}), D_{2,\infty}$\\\hline
  \multirow{3}{*}{(3)} & $SL_2(\F_3) \times C_2 $ & $\F_{25}$ & $\Q(\sqrt{-6}), \Q(\sqrt{-6})$\\\cline{2-4}
  & $SL_2(\F_3) \times C_4 $ & $\F_{25}$ & $\Q(\sqrt{-6}), \Q(\sqrt{-1})$\\\cline{2-4}
  & $SL_2(\F_3) \times C_6$ & $\F_{25}$ & $\Q(\sqrt{-6}), \Q(\sqrt{-3})$\\\hline
   \multirow{5}{*}{(4)} & $(C_{12} \rtimes C_2) \times C_2,~~ (\mathfrak{T}^* \rtimes C_4) \times C_2$ & $\F_{5}$ & $\Q(\sqrt{-1}), \Q(\sqrt{-11})$\\\cline{2-4}
   & $(C_{12} \rtimes C_2 )\times C_4, ~~(\mathfrak{T}^* \rtimes C_4) \times C_4$ & $\F_{5}$ & $\Q(\sqrt{-1}), \Q(\sqrt{-1})$\\\cline{2-4}
   &$(C_{12} \rtimes C_2) \times C_6, ~~(\mathfrak{T}^* \rtimes C_4) \times C_6$ & $\F_{25}$ & $\Q(\sqrt{-1}), \Q(\sqrt{-3})$\\\cline{2-4}
   & $ (C_{12} \rtimes C_2) \times \textrm{Dic}_{12}, ~~(\mathfrak{T}^* \rtimes C_4) \times \textrm{Dic}_{12}$ & $\F_{9}$ & $\Q(\sqrt{-1}), D_{3,\infty}$\\\cline{2-4}
   &$(C_{12} \rtimes C_2) \times \mathfrak{T}^*,~~(\mathfrak{T}^* \rtimes C_4) \times \mathfrak{T}^*$ & $\F_4$ &  $\Q(\sqrt{-1}), D_{2,\infty}$\\\hline
    \multirow{4}{*}{(5)} & $GL_2(\F_3) \times C_2$ & $\F_{17}$ & $\Q(\sqrt{-2}), \Q(\sqrt{-13})$\\\cline{2-4}
    & $GL_2(\F_3) \times C_4$ &$\F_{17}$ &$\Q(\sqrt{-2}), \Q(\sqrt{-1})$\\\cline{2-4}
    &$GL_2(\F_3) \times C_6 $& $\F_{9}$ & $\Q(\sqrt{-2}), \Q(\sqrt{-3})$\\\cline{2-4}
    & $GL_2(\F_3) \times \textrm{Dic}_{12}$ &$\F_{9}$ & $\Q(\sqrt{-2}), D_{3,\infty}$\\\hline
    \multirow{5}{*}{(6)}
    & $((C_6 \times C_6)\rtimes C_2) \times C_2,~~(\mathfrak{T}^* \times C_3) \times C_2$ & $\F_{7}$ & $\Q(\sqrt{-3}), \Q(\sqrt{-6})$ \\\cline{2-4}
    & $((C_6 \times C_6)\rtimes C_2) \times C_6,~~(\mathfrak{T}^* \times C_3) \times C_6$ & $\F_{7}$ & $\Q(\sqrt{-3}), \Q(\sqrt{-3})$ \\\cline{2-4}
    & $((C_6 \times C_6)\rtimes C_2) \times C_4,~~(\mathfrak{T}^* \times C_3) \times C_4$ & $\F_{25}$ & $\Q(\sqrt{-3}), \Q(\sqrt{-1})$ \\\cline{2-4}
    & $((C_6 \times C_6)\rtimes C_2) \times \textrm{Dic}_{12},~~(\mathfrak{T}^* \times C_3) \times \textrm{Dic}_{12}$ & $\F_{9}$ & $\Q(\sqrt{-3}), D_{3,\infty}$ \\\cline{2-4}
    & $((C_6 \times C_6)\rtimes C_2) \times \mathfrak{T}^*,~~(\mathfrak{T}^* \times C_3) \times \mathfrak{T}^*$ & $\F_{4}$ & $\Q(\sqrt{-3}), D_{2,\infty}$\\\hline
    \multirow{6}{*}{(7)}
    & for $G_1 \in \{\mathfrak{I}^*, \textrm{Dic}_{24} \}$ & & \\
    & $G_1 \times C_2 $ & \multirow{3}{*}{$\F_{49}$} & $D_{7,\infty}, \Q(\sqrt{-5})$\\
    & $G_1 \times C_4$ &   &$D_{7,\infty}, \Q(\sqrt{-1})$\\
    & $G_1 \times C_6$ &   &$D_{7,\infty}, \Q(\sqrt{-3})$\\\cline{2-4}
    & $\mathfrak{O}^* \times C_2 $ & \multirow{3}{*}{$\F_{121}$} & $D_{11,\infty}, \Q(\sqrt{-2})$\\
    & $\mathfrak{O}^* \times C_4$ &  &$D_{11,\infty}, \Q(\sqrt{-1})$\\
    & $\mathfrak{O}^* \times C_6$ &   &$D_{11,\infty}, \Q(\sqrt{-3})$\\\hline
     \multirow{5}{*}{(8)}
    & for $G_1 \in \{2^{1+4}_{-}.\textrm{Alt}_5, SL_{2}(\F_3) \times \textrm{Sym}_3, (SL_2(\F_3))^2 \rtimes \textrm{Sym}_2 \}$ & & \\
    & $G_1 \times C_2 $ & \multirow{4}{*}{$\F_{4}$} & $D_{2,\infty}, \Q(\sqrt{-15})$\\
    & $G_1 \times C_4$ &   &$D_{2,\infty}, \Q(\sqrt{-1})$\\
    & $G_1 \times C_6$ &   &$D_{2,\infty}, \Q(\sqrt{-3})$\\
    & $G_1 \times \mathfrak{T}^* $
    & & $D_{2,\infty}, D_{2,\infty}$\\\hline
     \multirow{5}{*}{(9)}
    & for $G_1 \in \{SL_2(\F_9), C_3 : (SL_2(\F_3).2), (\textrm{Dic}_{12})^2 \rtimes \textrm{Sym}_2 \}$ & & \\
    & $G_1 \times C_2 $ & \multirow{4}{*}{$\F_{9}$} & $D_{3,\infty}, \Q(\sqrt{-11})$\\
    & $G_1 \times C_4$ &   &$D_{3,\infty}, \Q(\sqrt{-1})$\\
    & $G_1 \times C_6$ &   &$D_{3,\infty}, \Q(\sqrt{-3})$\\
    & $G_1 \times \textrm{Dic}_{12}$
    & & $D_{3,\infty}, D_{3,\infty}$\\\hline
     \multirow{4}{*}{(10)}
    & for $G_1 \in \{SL_2(\F_5).2, SL_2(\F_5):2 \}$ & & \\
    & $G_1 \times C_2 $ & \multirow{3}{*}{$\F_{25}$} & $D_{5,\infty}, \Q(\sqrt{-6})$\\
    & $G_1 \times C_4$ &   &$D_{5,\infty}, \Q(\sqrt{-1})$\\
    & $G_1 \times C_6$ &  &$D_{5,\infty}, \Q(\sqrt{-3})$\\\hline
  \end{longtable}
This completes the proof.
\end{proof}

\begin{remark}\label{exclude rmk3}
Among all possible $99$ combinations of $G:=G_1 \times G_2$ with $G_1,G_2$ being as above, up to isomorphism, there are 19 groups which cannot be indeed realized as the automorphism group of a polarized abelian threefold over a finite field $k,$ which is maximal in the isogeny class of the product of a power of an elliptic curve and a non-isogenous elliptic curve over $k.$ For example, the groups $G:= G_1 \times \textrm{Dic}_{12} $ with $G_1 \in \{2_{-}^{1+4}.\textrm{Alt}_5, \textrm{SL}_2(\F_3) \times \textrm{Sym}_3, (\textrm{SL}_2(\F_3))^2 \rtimes \textrm{Sym}_2 \}$ cannot occur due to the issue of the characteristic of the base field $k.$
\end{remark}

For the case when the abelian threefold is a power of an ordinary elliptic curve over a finite field, we may use a result of \cite{BGL(1999)} to obtain the following theorem.

\begin{theorem}\label{powordelli}
  There exists a finite field $k$ and an ordinary elliptic curve $E$ over $k$ such that $G$ is the automorphism group of a polarized abelian threefold over $k,$ which is maximal in the isogeny class of $X:=E^3$ if and only if $G$ is one of the following groups (up to isomorphism): \\
(1) $D_6 \times C_2$; \\
(2) $C_2^3 \rtimes \textrm{Sym}_3$;\\
(3) $GL_3(\F_2) \times C_2 $; \\
(4) $[(Q_8 \rtimes C_3) \rtimes C_2] \times C_2 $;\\
(5) $(C_{12} \rtimes C_2 ) \times C_4 $;\\
(6) $C_4^3 \rtimes \textrm{Sym}_3 $; \\
(7) $ [(Q_8 \rtimes C_3) \cdot C_4] \times C_4 $;\\
(8) $[(Q_8 \rtimes C_3) \times C_3] \times C_6$;\\
(9) $C_6^3 \rtimes \textrm{Sym}_3$;\\
(10) $[([(C_3 \times C_6) \rtimes \textrm{Sym}_3 ] \cdot C_2 ) \cdot C_2] \cdot C_3$.
\end{theorem}
\begin{proof}
Suppose first that there exists a finite field $k$ and an ordinary elliptic curve $E$ over $k$ such that $G$ is the automorphism group of a polarized abelian threefold over $k,$ which is maximal in the isogeny class of $X:=E^3.$ In particular, we have $\textrm{End}_k^0(X)=M_3(\Q(\sqrt{-d}))$ for some square-free positive integer $d.$ Thus, by assumption, $G$ is a maximal finite subgroup of $GL_3(\Q(\sqrt{-d})),$ and then, by \cite[Theorems 6.1 and 7.1]{BGL(1999)}, we can see that $G$ must be one of the 10 groups in the statement of the theorem. Hence, it suffices to show the converse. We prove the converse by considering them one by one. We provide a detailed proof for (1), and then, we can proceed in a similar fashion with the specified choices for (2)-(10). \\
\indent (1) Take $G=D_6 \times C_2.$ Let $k=\F_{4}.$ Then there is an ordinary elliptic curve $E$ over $k$ such that $\textrm{End}_k^0(E)=\Q(\sqrt{-15})$ by \cite[Theorem 4.1]{12}. Let $X=E^3.$ Then we have $D:=\textrm{End}_k^0(X)=M_3(\Q(\sqrt{-15})).$ Let $\mathcal{O}=M_3 \left(\Z \left[\frac{1+\sqrt{-15}}{2} \right] \right).$ By Theorem~\ref{mat max}, $\mathcal{O}$ is a maximal $\Z$-order in $D,$ and hence, there exists an abelian threefold $X^{\prime}$ over $k$ such that $X^{\prime}$ is $k$-isogenous to $X$ and $\textrm{End}_k(X^{\prime})=\mathcal{O}$ by \cite[Theorem 3.13]{12}. In view of \cite[Theorem 7.1]{BGL(1999)}, $G$ is a maximal finite subgroup of $GL_3(\Q(\sqrt{-15})),$ and hence, $G$ is a maximal finite subgroup of $\mathcal{O}^{\times},$ too. \\
\indent   Now, let $\mathcal{L}$ be an ample line bundle on $X^{\prime}$, and put $\displaystyle \mathcal{L}^{\prime}:=\bigotimes_{f \in G} f^* \mathcal{L}.$ Then $\mathcal{L}^{\prime}$ is also an ample line bundle on $X^{\prime}$ that is preserved under the action of $G$ so that $G \leq \textrm{Aut}_k(X^{\prime},\mathcal{L}^{\prime}).$ Since $\textrm{Aut}_k (X^{\prime},\mathcal{L}^{\prime})$ is a finite subgroup of $\textrm{Aut}_k(X^{\prime})=\mathcal{O}^{\times},$ it follows from the maximality of $G$ that $G=\textrm{Aut}_k(X^{\prime},\mathcal{L}^{\prime}).$ \\
\indent  Now, suppose that $Y$ is an abelian threefold over $k$ which is $k$-isogenous to $X.$ In particular, we have $\textrm{End}_k^0(Y)=M_3(\Q(\sqrt{-15})).$ Suppose that there is a finite group $H$ such that $H=\textrm{Aut}_k(Y, \mathcal{M})$ for a polarization $\mathcal{M}$ on $Y,$ and $G$ is isomorphic to a proper subgroup of $H.$ Then $H$ is a finite subgroup of $GL_3(\Q(\sqrt{-15})),$ and hence, it follows from the maximality of $G$ as a finite subgroup of $GL_3(\Q(\sqrt{-15}))$ that $H=G$, which is a contradiction. Thus, we can conclude that $G$ is maximal in the isogeny class of $X.$ \\
\indent (2) For $G=C_2^3 \rtimes \textrm{Sym}_3,$ we take $k=\F_4$ and an ordinary elliptic curve $E$ over $k$ with $\textrm{End}_k^0(E)=\Q(\sqrt{-15}).$ \\
\indent (3) For $G=GL_3(\F_2) \times C_2$, we take $k=\F_{4}$ and an ordinary elliptic curve $E$ over $k$ with $\textrm{End}_k^0(E)=\Q(\sqrt{-7}).$ \\
\indent (4) For $G=[(Q_8 \rtimes C_3) \rtimes C_2] \times C_2,$ we take $k=\F_{17}$ and an ordinary elliptic curve $E$ over $k$ with $\textrm{End}_k^0(E)=\Q(\sqrt{-2}).$  \\
\indent (5)-(7) For the groups $G=(C_{12} \rtimes C_2 ) \times C_4, C_4^3 \rtimes \textrm{Sym}_3,$ or $ [(Q_8 \rtimes C_3) \cdot C_4] \times C_4 ,$ we take $k=\F_{5}$ and an ordinary elliptic curve $E$ over $k$ with $\textrm{End}_k^0(E)=\Q(\sqrt{-1}).$ \\
\indent (8)-(10) For the groups $G=[(Q_8 \rtimes C_3) \times C_3] \times C_6, C_6^3 \rtimes \textrm{Sym}_3,$ or $ [([(C_3 \times C_6) \rtimes \textrm{Sym}_3 ] \cdot C_2 ) \cdot C_2] \cdot C_3,$ we take $k=\F_{7}$ and an ordinary elliptic curve $E$ over $k$ with $\textrm{End}_k^0(E)=\Q(\sqrt{-3}).$\\
\indent This completes the proof.
\end{proof}

 Finally, in the next theorem, we consider the case when the abelian threefold is a power of a supersingular elliptic curve over a finite field.
\begin{theorem}\label{pow of supell}
  There exists a finite field $k$ and a supersingular elliptic curve $E$ over $k$ (all of whose endomorphisms are defined over $k$) such that $G$ is the automorphism group of a polarized abelian threefold over $k,$ which is maximal in the isogeny class of $X:=E^3$ if and only if $G$ is one of the following groups (up to isomorphism): \\
(1) $SL_2(\F_5) $;\\
(2) $(SL_2(\F_3))^3 \rtimes \textrm{Sym}_3 $; \\
(3) $\pm U_3(3) $; \\
(4) $\pm 3_{+}^{1+2}.GL_2(\F_3) $;\\
(5) $(\textrm{Dic}_{12})^3 \rtimes \textrm{Sym}_3 $;\\
(6) $\pm L_2(7).2 $;\\
(7) $GL_3(\F_2) \times C_2  $;\\
(8) $ C_4^3 \rtimes \textrm{Sym}_3 \cong C_4 \wr \textrm{Sym}_3 $; \\
(9) $C_6^3 \rtimes \textrm{Sym}_3 \cong C_6 \wr \textrm{Sym}_3$; \\
(10) $[([(C_3 \times C_6) \rtimes \textrm{Sym}_3 ] \cdot C_2) \cdot C_2 ]\cdot C_3 \cong (\textit{He}_3 \times \langle -\textbf{1}_3 \rangle ) \rtimes SL_2(\F_3) $; \\
(11) $\textrm{Dic}_{28} $;\\ 
(12) $C_2^3 \rtimes \textrm{Sym}_3 \cong C_2 \wr \textrm{Sym}_3$;\\
(13) $2_{-}^{1+4}.\textrm{Alt}_5 \times \mathfrak{T}^*$ and $(SL_2(\F_3 ) \times \textrm{Sym}_3) \times \mathfrak{T}^*$;\\
(14) $SL_2(\F_9) \times \textrm{Dic}_{12}, C_3 : (SL_2(\F_3).2) \times \textrm{Dic}_{12}$, and $(\mathfrak{T}^* \rtimes C_4) \times \textrm{Dic}_{12}$;\\
(15) $(SL_2(\F_5).2) \times C_6, (SL_2(\F_5):2) \times C_6,$ and $(\mathfrak{T}^* \times C_3) \times C_6$; \\
(16) $GL_2(\F_3) \times C_2, \mathfrak{I}^* \times C_2,$ and $\mathfrak{O}^* \times C_2$;\\
(17) $GL_2(\F_3) \times C_4, (C_{12} \rtimes C_2) \times C_4, (\mathfrak{T}^* \rtimes C_4) \times C_4, \textrm{Dic}_{24} \times C_4,$ and $\mathfrak{O}^* \times C_4$;\\
(18) $GL_2(\F_3) \times C_6, \mathfrak{O}^* \times C_6, \textrm{Dic}_{24}\times C_6,$ and $\mathfrak{I}^* \times C_6$;\\
(19) $(\mathfrak{T}^* \rtimes C_4) \times C_6$; \\
(20) $(\mathfrak{T}^* \times C_3) \times C_4$ and $(\textrm{Dic}_{12} \rtimes C_6) \times C_4$; \\
(21) $\mathfrak{I}^* \times C_4$;\\
(22) $D_6 \times C_2, \mathfrak{T}^* \times C_2,$ and $\textrm{Dic}_{12} \times C_2.$
\end{theorem}

\begin{proof}
Suppose first that there exists a finite field $k=\F_q$ ($q=p^a$) and a supersingular elliptic curve $E$ over $k$ (all of whose endomorphisms are defined over $k$) such that $G$ is the automorphism group of a polarized abelian threefold over $k,$ which is maximal in the isogeny class of $X:=E^3.$ In particular, we have $\textrm{End}_k^0(X)=M_3(D_{p,\infty}).$ Thus, by assumption, $G$ is a maximal finite subgroup of $GL_3(D_{p,\infty})$ (up to isomorphism). If $G$ is absolutely irreducible, then $G$ must be one of the 6 groups $(1)\sim(6)$ in the above list by Theorems~\ref{prim aimf of GL3} and~\ref{imprim aimf GL2}. If $G$ is irreducible, but not absolutely irreducible, then, as we have seen in Section~\ref{quat mat rep}, $G$ must be one of the 6 groups ($7)\sim(12$) in the above list. If $G$ is reducible, then by Theorems~\ref{type I red new} and~\ref{type II red}, $G$ is one of the groups in cases $(13) \sim (22)$. Hence, it suffices to show the converse. We prove the converse by considering them one by one. \\

Groups (1)-(6): We provide a detailed proof for (1), and then, we can proceed in a similar fashion with the specified choices for (2)-(6). \\
\indent (1) Let $k=\F_{4}.$ Then there is a supersingular elliptic curve $E$ over $k$ such that $\textrm{End}_k^0(E)=D_{2,\infty}$ by \cite[Theorem 4.1]{12}. Let $X=E^3.$ Then we have $D:=\textrm{End}_k^0(X)=M_3(D_{2,\infty}).$ Let $V=D_{2,\infty}^3, \mathcal{O} = \Z \left[i, j, ij, \frac{1+i+j+ij}{2} \right],$ and $G=SL_2(\F_5).$ By Theorem~\ref{prim aimf of GL3}, $G$ is a primitive absolutely irreducible maximal finite subgroup of $GL_3 (D_{2,\infty}).$ Recall also that $\mathcal{O}$ is a maximal $\Z$-order in $D_{2,\infty}.$ By \cite[Definition and Lemma 2.6]{8}, there is a $G$-invariant $\mathcal{O}$-lattice $L$ in $V,$ and then, by Theorem~\ref{mat max 2}, it follows that $\mathcal{O}^{\prime}:=\textrm{Hom}_{\mathcal{O}}(L,L)$ is a maximal $\Z$-order in $\textrm{Hom}_{D_{2,\infty}}(V,V)=M_3(D_{2,\infty}).$ By the choice of $L,$ $G$ can be regarded as a subgroup of $(\mathcal{O}^{\prime})^{\times}.$ Also, by \cite[Theorem 3.13]{12}, there exists an abelian threefold $X^{\prime}$ over $k$ such that $X^{\prime}$ is $k$-isogenous to $X$ and $\textrm{End}_k(X^{\prime})=\mathcal{O}^{\prime}$. Then it follows that $G \leq \textrm{Aut}_k(X^{\prime}).$ \\
\indent Now, let $\mathcal{L}$ be an ample line bundle on $X^{\prime}$, and put $\displaystyle \mathcal{L}^{\prime}:=\bigotimes_{f \in G} f^* \mathcal{L}.$ Then $\mathcal{L}^{\prime}$ is also an ample line bundle on $X^{\prime}$ that is preserved under the action of $G$ so that $G \leq \textrm{Aut}_k(X^{\prime},\mathcal{L}^{\prime}).$ Since $\textrm{Aut}_k (X^{\prime},\mathcal{L}^{\prime})$ is a finite subgroup of $\textrm{Aut}_k(X^{\prime})=(\mathcal{O}^{\prime})^{\times},$ it follows from the maximality of $G$ that $G=\textrm{Aut}_k(X^{\prime},\mathcal{L}^{\prime}).$ Furthermore, by a similar argument, we can see that $G$ is maximal in the isogeny class of $X.$  \\
\indent (2) For the group $G=(SL_2(\F_3))^3 \rtimes \textrm{Sym}_3$, we can proceed as above with the choices $k=\F_4$ and a supersingular elliptic curve $E$ over $k$ with $\textrm{End}_k^0(E)=D_{2,\infty}.$ Here, we recall that $G$ is an imprimitive absolutely irreducible maximal finite subgroup of $GL_3(D_{2,\infty}).$ Similarly: \\
\indent (3)-(5) For the groups $G=\pm U_3(3)$ or $G=\pm 3_{+}^{1+2}.GL_2(\F_3)$ or $G=(\textrm{Dic}_{12})^3 \rtimes \textrm{Sym}_3,$ we take $k=\F_{9}$, a supersingular elliptic curve $E$ over $k$ with $\textrm{End}_k^0(E)=D_{3,\infty}$, and a (unique) maximal $\Z$-order $\mathcal{O}$ in $D_{3,\infty}$. Here, we recall that the first two groups (resp.\ the last group) are primitive (resp.\ imprimitive) absolutely irreducible maximal finite subgroups of $GL_3(D_{3,\infty}).$  \\
\indent (6) For the group $G=\pm L_2(7).2,$ we take $k=\F_{49}$, a supersingular elliptic curve $E$ over $k$ with $\textrm{End}_k^0(E)=D_{7,\infty}$, and a (unique) maximal $\Z$-order $\mathcal{O}$ in $D_{7,\infty}$. Here, we recall that $G$ is a primitive absolutely irreducible maximal finite subgroup of $GL_3(D_{7,\infty}).$  \\

Groups (7)-(12): We provide a detailed proof for (7), and then, we can proceed in a similar fashion with the specified choices for (8)-(12). \\
\indent (7) Let $k=\F_{169}.$ Then there is a supersingular elliptic curve $E$ over $k$ such that $\textrm{End}_k^0(E)=D_{13,\infty}$ by \cite[Theorem 4.1]{12}. Let $X=E^3.$ Then we have $D:=\textrm{End}_k^0(X)=M_3(D_{13,\infty}).$ Let $V=D_{13,\infty}^3, \mathcal{O}$ a maximal $\Z$-order in $D_{13,\infty}$, and $G=GL_3(\F_2) \times C_2.$ By Lemma~\ref{imag lem} and Example~\ref{GL(2,3)}, $G$ is an irreducible maximal finite subgroup of $GL_3 (D_{13,\infty}).$ By \cite[Definition and Lemma 2.6]{8}, there is a $G$-invariant $\mathcal{O}$-lattice $L$ in $V,$ and then, by Theorem~\ref{mat max 2}, it follows that $\mathcal{O}^{\prime}:=\textrm{Hom}_{\mathcal{O}}(L,L)$ is a maximal $\Z$-order in $\textrm{Hom}_{D_{13,\infty}}(V,V)=M_3(D_{13,\infty}).$ Then by a similar argument as in the proof of (1), we can see that there exists an abelian threefold $X^{\prime}$ over $k$ (being $k$-isogenous to $X$) with a polarization $\mathcal{L}^{\prime}$ such that $G=\textrm{Aut}_k(X^{\prime}, \mathcal{L}^{\prime})$ and $G$ is maximal in the isogeny class of $X.$ \\
\indent (8)-(10) For the groups $G=C_4^3 \rtimes \textrm{Sym}_3$ or $G=C_6^3 \rtimes \textrm{Sym}_3$ or $G=[([(C_3 \times C_6) \rtimes \textrm{Sym}_3 ] \cdot C_2) \cdot C_2 ]\cdot C_3,$ we take $k=\F_{121}$, a supersingular elliptic curve $E$ over $k$ with $\textrm{End}_{k}^0(E)=D_{11,\infty}$, and a maximal $\Z$-order $\mathcal{O}$ in $D_{11,\infty}$. Here, we recall that those groups are irreducible maximal finite subgroups of $GL_3(D_{11,\infty})$ (see Example~\ref{GL(2,3)}). \\
\indent (11) For the group $G=\textrm{Dic}_{28}$, we take $k= \F_{49}$, a supersingular elliptic curve $E$ over $k$ with $\textrm{End}_k^0(E)=D_{7,\infty}$, and a maximal $\Z$-order $\mathcal{O}$ in $D_{7,\infty}.$ Here, we recall that $G$ is an irreducible maximal finite subgroup of $GL_3(D_{7,\infty})$ (see Lemma~\ref{div alg lem}).  \\
\indent (12) For the group $G=C_2 \wr \textrm{Sym}_3,$ we take $k=\F_{11881}$, a supersingular elliptic curve $E$ over $k$ with $\textrm{End}_k^0(E)=D_{109,\infty}$, and a maximal $\Z$-order $\mathcal{O}$ in $D_{109,\infty}.$ Here, we recall that $G$ is an irreducible maximal finite subgroup of $GL_3(D_{109,\infty})$  (see Lemma~\ref{irre max Q mat}). \\

Groups in (13)-(22): We provide a detailed proof for the groups in case (13), and then, we can proceed in a similar fashion with the specified choices for the groups in cases (14)-(22). \\
\indent (13) Let $k=\F_{4}.$ Then there is a supersingular elliptic curve $E$ over $k$ such that $\textrm{End}_k^0(E)=D_{2,\infty}$ by \cite[Theorem 4.1]{12}. Let $X=E^3.$ Then we have $D:=\textrm{End}_k^0(X)=M_3(D_{2,\infty}).$ Let $V=D_{2,\infty}^3, \mathcal{O} = \Z \left[i, j, ij, \frac{1+i+j+ij}{2} \right],$ and $G=2_{-}^{1+4}.\textrm{Alt}_5 \times \mathfrak{T}^*$ or $G=(SL_2(\F_3) \times \textrm{Sym}_3) \times \mathfrak{T}^*.$ By Theorem~\ref{type I red new}, $G$ is a reducible maximal finite subgroup of $GL_3 (D_{2,\infty}).$ Recall also that $\mathcal{O}$ is a maximal $\Z$-order in $D_{2,\infty}.$ By \cite[Definition and Lemma 2.6]{8}, there is a $G$-invariant $\mathcal{O}$-lattice $L$ in $V,$ and then, by Theorem~\ref{mat max 2}, it follows that $\mathcal{O}^{\prime}:=\textrm{Hom}_{\mathcal{O}}(L,L)$ is a maximal $\Z$-order in $\textrm{Hom}_{D_{2,\infty}}(V,V)=M_3(D_{2,\infty}).$ Then by a similar argument as in the proof of (1), we can see that there exists an abelian threefold $X^{\prime}$ over $k$ (being $k$-isogenous to $X$) with a polarization $\mathcal{L}^{\prime}$ such that $G=\textrm{Aut}_k(X^{\prime}, \mathcal{L}^{\prime})$ and $G$ is maximal in the isogeny class of $X.$ \\
\indent (14) For the groups in (14), we take $k=\F_9$, a supersingular elliptic curve $E$ over $k$ with $\textrm{End}_k^0(E)=D_{3,\infty},$ and a maximal $\Z$-order $\mathcal{O}$ in $D_{3,\infty}.$ Here, we recall that those groups are reducible maximal finite subgroups of $GL_3(D_{3,\infty}).$ \\
\indent (15) For the groups in (15), we take $k=\F_{25}$, a supersingular elliptic curve $E$ over $k$ with $\textrm{End}_k^0(E)=D_{5,\infty},$ and a maximal $\Z$-order $\mathcal{O}$ in $D_{5,\infty}.$ Here, we recall that those groups are reducible maximal finite subgroups of $GL_3(D_{5,\infty}).$ \\
\indent (16) For the groups in (16), we take $k=\F_{1369}$, a supersingular elliptic curve $E$ over $k$ with $\textrm{End}_k^0(E)=D_{37,\infty},$ and a maximal $\Z$-order $\mathcal{O}$ in $D_{37,\infty}.$ Here, we recall that those groups are reducible maximal finite subgroups of $GL_3(D_{37,\infty}).$ \\
\indent (17) For the groups $G=GL_2(\F_3) \times C_4,$ $G=(C_{12} \rtimes C_2)\times C_4$, $G=(\mathfrak{T}^* \rtimes C_4) \times C_4$, or $G=\textrm{Dic}_{24} \times C_4$, we take $k=\F_{961}$, a supersingular elliptic curve $E$ over $k$ with $\textrm{End}_k^0(E)=D_{31,\infty},$ and a maximal $\Z$-order $\mathcal{O}$ in $D_{31,\infty}.$ Here, we recall that those groups are reducible maximal finite subgroups of $GL_3(D_{31,\infty}).$ \\
\indent For the group $G=\mathcal{O}^* \times C_4,$ we take $k=\F_{361}$, a supersingular elliptic curve $E$ over $k$ with $\textrm{End}_k^0(E)=D_{19,\infty},$ and a maximal $\Z$-order $\mathcal{O}$ in $D_{19,\infty}.$ Here, we recall that $G$ is a reducible maximal finite subgroup of $GL_3(D_{19,\infty}).$ \\
\indent (18) For the groups in (18), we take $k=\F_{2809}$, a supersingular elliptic curve $E$ over $k$ with $\textrm{End}_k^0(E)=D_{53,\infty},$ and a maximal $\Z$-order $\mathcal{O}$ in $D_{53,\infty}.$ Here, we recall that those groups are reducible maximal finite subgroups of $GL_3(D_{53,\infty}).$ \\
\indent (19)-(20) For the groups in (19) and (20), we take $k=\F_{5041}$, a supersingular elliptic curve $E$ over $k$ with $\textrm{End}_k^0(E)=D_{71,\infty},$ and a maximal $\Z$-order $\mathcal{O}$ in $D_{71,\infty}.$ Here, we recall that those groups are reducible maximal finite subgroups of $GL_3(D_{71,\infty}).$ \\
\indent (21) For the group $G=\mathfrak{I}^* \times C_4$, we take $k=\F_{1849}$, a supersingular elliptic curve $E$ over $k$ with $\textrm{End}_k^0(E)=D_{43,\infty},$ and a maximal $\Z$-order $\mathcal{O}$ in $D_{43,\infty}.$ Here, we recall that $G$ is a reducible maximal finite subgroup of $GL_3(D_{43,\infty}).$ \\
\indent (22) For the groups in (22), we take $k=\F_{58081}$, a supersingular elliptic curve $E$ over $k$ with $\textrm{End}_k^0(E)=D_{241,\infty},$ and a maximal $\Z$-order $\mathcal{O}$ in $D_{241,\infty}.$ Here, we recall that those groups are reducible maximal finite subgroups of $GL_3(D_{241,\infty}).$ \\
\indent This completes the proof.
\end{proof}

\section*{Acknowledgement}
WonTae Hwang was supported by a KIAS individual Grant (MG069901) at Korea Institute for Advanced Study. Bo-Hae Im was supported by Basic Science Research Program through the National Research Foundation of Korea(NRF) funded by the Ministry of  Science \& ICT(NRF-2020R1A2B5B01001835). Bo-Hae Im would also like to thank KIAS (Korea Institute for Advanced Study) for its hospitality.


\end{document}